\documentclass{amsart}

\usepackage{latexsym,amssymb,amsmath, amsfonts}
\usepackage{amsthm}
\usepackage{amscd}

\newcommand{\R}{\mathbb{R}}

\def\trans{\mathrm{transversal}}
\def\rest{\mathrm{Rest}}
\def\osc{\mathrm{Osc}}
\def\kak{\mathrm{Kak}}
\def\ckak{\mathrm{Curvy Kak}}
\def\hess{\mathrm{Hess}}

\newcommand{\be}{\begin{equation}}
\newcommand{\en}{\end{equation}}
\newcommand{\ee}{\end{equation}}

\DeclareMathOperator{\supp}{supp}
\newcommand{\bt}{\begin{theorem}}
\newcommand{\et}{\end{theorem}}
\newcommand{\bp}{\begin{proof}}
\newcommand{\ep}{\end{proof}}
\newcommand{\bc}{\begin{cor}}
\newcommand{\ec}{\end{cor}}
\newcommand{\bl}{\begin{lemma}}
\newcommand{\el}{\end{lemma}}
\newcommand{\bprop}{\begin{prop}}
\newcommand{\eprop}{\end{prop}}

\def\R{{\hbox{\bf R}}}

\def\ex{{\mathcal{E}}}

\def\tubes{{\mathbb{T}}}

\newtheorem{theorem}{Theorem}[section]

\newtheorem{lemma}[theorem]{Lemma}
\newtheorem{definition}{Definition}
\newtheorem{prop}[theorem]{Proposition}
\newtheorem{conjecture}{Conjecture}

\newtheorem{cor}[theorem]{Corollary}
\newtheorem*{thma}{Tentative Conjecture}
\numberwithin{theorem}{section} \numberwithin{definition}{section}

\theoremstyle{definition}

\begin{document}

\title{Aspects of Multilinear Harmonic Analysis Related to Transversality}


\author{Jonathan Bennett}
\address{School of Mathematics \\
The University of Birmingham \\
The Watson Building \\
Edgbaston \\
Birmingham \\
B15 2TT \\
United Kingdom}
\email{J.Bennett@bham.ac.uk}
\thanks{Supported by ERC Starting Grant
307617. The author would like to thank the organisers of this most stimulating El Escorial conference for the invitation to deliver a short course on the material presented here.
First published in Contemporary Mathematics in 2014, published by the American Mathematical Society. \copyright2014 American Mathematical Society.
}

\subjclass[2000]{}


\begin{abstract}
The purpose of this article is to survey certain aspects of multilinear harmonic analysis related to notions of transversality. Particular emphasis will be placed on the multilinear restriction theory for the euclidean Fourier transform, multilinear oscillatory integrals, multilinear geometric inequalities, multilinear Radon-like transforms, and the interplay between them.
\end{abstract}

\maketitle


\section{Introduction}




In the 1970's fundamental work of Fefferman and Stein (such as \cite{Feff}) led to a dramatic change of perspective in euclidean harmonic analysis, placing elementary geometric concepts such as \emph{curvature} at the heart of the subject.
As a result many of its core open problems today concern phenomena where the presence of some underlying curved manifold plays a fundamental role.
One of the most important examples is the celebrated and far-reaching \emph{restriction conjecture for the Fourier transform},
which concerns the size of the restriction of the Fourier transform of an $L^p$ function to a curved submanifold of euclidean space. Further important examples include the Bochner--Riesz conjecture, combinatorial problems of \emph{Kakeya type}, and size and smoothing estimates for \emph{Radon-like transforms} and their singular and maximal variants -- see the articles \cite{SW} and \cite{SteinICM} for further discussion of these objects and this enduring perspective. As Stein points out in \cite{SteinICM}, such curvature-related problems tend to be intimately related to the theory of oscillatory integrals.


The last decade or so has seen the emergence of a wide-ranging \emph{multilinear} perspective on many of the central elements of modern harmonic analysis. This has involved the establishment of multilinear variants of classical objects and methodologies, such as singular integrals, multiplier theorems, maximal operators, weighted inequalities and interpolation (see \cite{MS}).
The purpose of this article is to describe aspects of this emerging multilinear perspective in the setting of problems related to curvature. As will become apparent, in a multilinear setting the relevance of curvature is often diminished, being naturally replaced by notions of \emph{transversality}. This fundamental change of nature (from a second order hypothesis to a first order hypothesis) turns out to be most significant in its impact upon the methodologies at our disposal. Perhaps the most striking example is the algebraic-geometric approach to the endpoint multilinear Kakeya conjecture developed by Guth in \cite{G}. We do not touch on this here,\footnote{We refer the reader to \cite{G} and \cite{CV} for a full discussion.} but instead emphasise the extraordinary effectiveness of the induction-on-scales method in multilinear settings.
As we shall see, such inductive arguments, which in this context have their origins in Bourgain \cite{Bo}, function particularly well in the setting of inequalities which are subordinate to transversality rather than curvature hypotheses.

Very recently Bourgain and Guth \cite{BG} have developed a method for applying transversal multilinear inequalities to classical curvature-related linear problems. This has provided a mechanism through which unexpected techniques may be brought to bear on difficult curvature-related problems in harmonic analysis and PDE; see \cite{BourgainMoment}, \cite{BourgainSch}, \cite{BourgainLap},\cite{BH} and \cite{LV}. In particular, this is responsible for the current state-of-the-art on the restriction conjecture for the Fourier transform. We do not attempt to properly survey these applications in this article.

This article is organised as follows. In Sections \ref{two}-\ref{MultSect} we discuss aspects of the recent multilinear perspective on the restriction conjecture for the Fourier transform, emphasising its relation with multilinear geometric and combinatorial (Kakeya) inequalities. In the final section (Section \ref{five}) we describe a bigger picture which ultimately seeks a broad multilinear analogue of the classical H\"ormander theorem on the $L^2$-boundedness of nondegenerate oscillatory integral operators. As may be expected, such multilinear oscillatory integral results are closely related to more manifestly geometric problems, such as bounds on multilinear Radon-like transforms.

This article is not intended to be a comprehensive survey of the area. Emphasis will be placed on collaborative work of the author, and in particular an exposition and reworking of aspects of B--Carbery--Tao \cite{BCT}. Mainly for reasons of space, we do not attempt to discuss the above-mentioned applications of these results in any detail.

\section{The linear restriction theory for the Fourier transform}\label{two}
There are of course other articles which survey aspects of the restriction theory for the Fourier transform; see, in particular, Tao's treatment \cite{TaoRestn}.
\subsection{The classical restriction conjecture and early progress}\label{twoone}
For $d\geq 2$, let $U$ be a compact neighbourhood of the origin in
$\mathbb{R}^{d-1}$ and $\Sigma :U\rightarrow \mathbb{R}^{d}$ be a smooth
parametrisation of a $(d-1)$-dimensional submanifold $S$ of $\mathbb{R}^d$
(for instance, $S$ could be a small portion of the unit sphere
$\mathbb{S}^{d-1}$, paraboloid or hyperplane). With $\Sigma$ we associate the \textit{extension
operator} $\ex$, given by
\begin{equation}\label{extn}
\ex g(\xi):=\int_{U}g(x)e^{i\xi\cdot\Sigma(x)}dx,
\end{equation}
where $g\in L^{1}(U)$ and $\xi\in\mathbb{R}^d$. This operator is sometimes
referred to as the \textit{adjoint restriction operator} since its
adjoint $\ex^{*}$ is given by $\ex^{*}f=\widehat{f}\circ\Sigma$,
where $\widehat{\;}\:$ denotes the $d$-dimensional Fourier
transform.  In addition to their intrinsic interest in harmonic analysis, Fourier extension operators are central to the study of dispersive partial differential equations. This is apparent on observing that if $\Sigma(x)=(x,|x|^2/2)$ (so that $S$ is a paraboloid), then $\mathcal{E}\widehat{u}_0(x,t)$ solves the Schr\"odinger equation $i\partial_t u+\Delta u=0$ with initial data $u(\cdot, 0)=u_0$. Here $(x,t)\in\mathbb{R}^{d-1}\times\mathbb{R}$. 

At this level of generality there are no $L^{p}(U) \to
L^{q}(\mathbb{R}^{d})$ estimates for $\ex$ other than the trivial
$$\|\mathcal{E}g\|_{L^\infty(\mathbb{R}^d)}\leq\|g\|_{L^1(U)};$$ a fact that is immediate upon taking $S$ to be a portion of (for example) the $d$th coordinate hyperplane and $\Sigma(x)=(x_1,\hdots, x_{d-1},0)$, since in this case we have $\mathcal{E}g(\xi)=\widehat{g}(\xi_1,\hdots,\xi_{d-1})$, which is of course independent of the component $\xi_d$. Thus if $g\not\equiv 0$ then $\mathcal{E}g\in L^q(\mathbb{R}^d)$ only if $q=\infty$.

However, as was first observed by E. M. Stein in the late 1960's,
if the submanifold $S$ has everywhere
non-vanishing gaussian curvature, then non-trivial $L^{p}(U) \to
L^{q}(\mathbb{R}^{d})$ estimates for $\ex$ may be obtained. This may be seen rather easily using the ``$TT^*$ method" along with the key stationary-phase fact that (if $S$ has non-vanishing curvature) the Fourier transform of a smooth density supported on $S$ has sufficient decay to belong to $L^p(\mathbb{R}^d)$ for some $p<\infty$. More precisely, an elementary calculation reveals that $\mathcal{E}\mathcal{E}^{*}f=\widehat{\mu}*f$, where $\mu$ is the $S$-carried measure defined by
\begin{equation}\label{mu}
\int \phi d\mu=\int_U\phi(\Sigma(x))dx,
\end{equation}
and so by Young's convolution inequality, $\|\mathcal{E}\mathcal{E}^*f\|_q\leq\|\widehat{\mu}\|_{q/2}\|f\|_{q'}$.
Combining this with the well-known stationary-phase estimate
\begin{equation}\label{decay}
|\widehat{\mu}(\xi)|\lesssim (1+|\xi|)^{-\frac{d-1}{2}}
\end{equation}
reveals that $\|\mathcal{E}\mathcal{E}^*f\|_q\leq\|f\|_{q'}$, and thus
\begin{equation}\label{L2rest}
\|\mathcal{E}g\|_{q}\lesssim\|g\|_2,
\end{equation}
provided $q>\frac{4d}{d-1}$. A refinement of these ideas leads to an improvement to $q\geq 2(d+1)/(d-1)$ in the $L^2$ estimate \eqref{L2rest}. This is known as the Stein--Tomas restriction theorem. In a dispersive PDE setting, the inequality \eqref{L2rest} is an example of a \emph{Strichartz estimate}.

It is often helpful to use a more manifestly geometric interpretation of the extension operator $\mathcal{E}$. In particular, by the definition \eqref{mu} of the $S$-carried measure $\mu$ we have $\mathcal{E}g=\widehat{fd\mu}$ where $f\in L^1(d\mu)$ is the ``lift" of $g$ onto $S$, given by $f\circ\Sigma=g$. The inequalities
\begin{equation}\label{parsurf}
\|\ex g\|_{L^{q}(\mathbb{R}^d)}\lesssim\|g\|_{L^{p}(U)}
\end{equation}
and
$$\|\widehat{fd\mu}\|_{L^q(\mathbb{R}^d)}\lesssim\|f\|_{L^p(d\mu)}$$ are thus identical, and in turn,
equivalent to
\begin{equation}\label{geominterp}
\|\widehat{fd\sigma}\|_{L^q(\mathbb{R}^d)}\lesssim\|f\|_{L^p(d\sigma)},
\end{equation}
where, as usual, $d\sigma$ denotes the induced Lebesgue measure on $S$.
Of course the implicit constant in \eqref{geominterp} should no longer depend on our particular parametrisation $\Sigma$ of $S$.
We refer to the linear operator
$
f\mapsto\widehat{fd\sigma}
$
as the \emph{extension operator associated with $S$}.

The classical
\textit{restriction conjecture} concerns the full range of
exponents $p$ and $q$ for which such bounds hold.
\begin{conjecture}[Linear Restriction]\label{LRC}
If $S$ has everywhere non--vanishing gaussian curvature,
$\tfrac{1}{q}<\tfrac{d-1}{2d}$ and $\tfrac{1}{q}
\leq\tfrac{d-1}{d+1}\tfrac{1}{p'}$, then
\begin{equation}\label{lrc}
\|\widehat{fd\sigma}\|_{L^q(\mathbb{R}^d)}\lesssim\|f\|_{L^p(d\sigma)},
\end{equation}
for all $g\in L^p(d\sigma)$.
\end{conjecture}
This conjecture was settled for $d=2$ by Fefferman \cite{Feffrest} and Zygmund \cite{Z} in the early 1970s. In higher dimensions, the case $p=2$ (and thus $q\geq 2(d+1)/(d-1)$) is the content of the Stein--Tomas restriction theorem. There has been significant further progress by Bourgain \cite{Bo}, \cite{BoS}, Wolff \cite{WolffRest}, Moyua, Vargas, Vega and Tao \cite{MVV}, \cite{TVV}, \cite{TV}, \cite{TaoKM}. The most recent progress is due to Bourgain and Guth \cite{BG} (see also Temur \cite{Temur}), and uses the multilinear restriction theory of Carbery, Tao and the author \cite{BCT}; see the forthcoming Section \ref{MultSect}. See \cite{TaoRestn} for further historical detail.


The restriction conjecture is generated by testing \eqref{parsurf} on characteristic functions of small balls in $\mathbb{R}^{d-1}$, or equivalently, testing \eqref{lrc} on characteristic functions of small ``caps" on $S$. More specifically, if $f=\chi_\rho$, where $\rho$ is a cap of diameter $0<\delta\ll 1$, centred at a point $x_\rho\in S$, then
$$|\widehat{fd\sigma}(\xi)|=\Bigl|\int_\rho e^{ix\cdot\xi}d\sigma(x)\Bigr|=\Bigl|\int_\rho e^{i(x-x_\rho)\cdot\xi}d\sigma(x)\Bigr|\gtrsim
\delta^{d-1}\chi_{\rho^*}(\xi),$$
where $$\rho^*=\{\xi\in\mathbb{R}^d: |(x-x_\rho)\cdot\xi|\leq 1\;\;\mbox{ for all }\;\;x\in \rho\}.$$
The set $\rho^*$ is a certain dual object to $\rho$, containing a rectangular tube of the form $O(\delta^{-2})T$, where $T$ has $d-1$ short sides of length $\delta$ and one long side of length $1$ pointing in the direction normal to $S$ at $x_\rho$. We refer to $T$ as a $\delta$-\emph{tube}.
Since $S$ has nonvanishing curvature, $|\rho^*|\sim \delta^{-2d}|T|\sim\delta^{-(d+1)}$.
Now, if \eqref{lrc} holds then
\begin{equation}\label{necknapp}
\delta^{d-1}|\rho^*|^{1/q}\lesssim |\rho|^{1/p}
\end{equation}
uniformly in $\delta$,
and so $$\delta^{d-1}\delta^{-(d+1)/q}\lesssim\delta^{(d-1)/p}$$ uniformly in $\delta$.
Letting $\delta\rightarrow 0$ forces the claimed necessary condition $\tfrac{1}{q}
\leq\tfrac{d-1}{d+1}\tfrac{1}{p'}$. The remaining condition $\tfrac{1}{q}<\tfrac{d-1}{2d}$ is an integrability condition, and is a manifestation of the optimality of the decay estimate \eqref{decay}; see \cite{TaoRestn} for further details.

Progress on the restriction conjecture beyond the Stein--Tomas exponent $q= 2(d+1)/(d-1)$ has required techniques that are much more geometric, going beyond what the decay estimates for $\widehat{\mu}$ (or equivalently $\widehat{\sigma}$) allow. These advances, originating in Bourgain \cite{Bo}, relied upon a compelling interplay between the restriction conjecture and the celebrated \emph{Kakeya conjecture} from geometric combinatorics.
\subsection{Relation with the classical Kakeya conjecture}
The above example involving a $\delta$-cap on $S$, may be developed much further by considering input functions $f$ formed by summing many (modulated) characteristic functions of disjoint $\delta$-caps $\{\rho\}$ on $S$. \footnote{Observe that if $S$ is, say, a patch of paraboloid then the resulting dual tubes $\rho^*$ have different orientations. If no curvature condition is imposed on $S$ then the resulting tubes may be parallel (or coincident) and, furthermore, have arbitrary length. These considerations will become relevant later.} While the $L^p(d\sigma)$ norm of such a sum is straightforward to compute, estimates on the left-hand-side of \eqref{lrc} turn out to be very difficult. Indeed a standard randomisation (or Rademacher function) argument, originating in \cite{Feff} and \cite{BCSS}, reveals that the restriction conjecture implies the following form of the Kakeya conjecture from geometric combinatorics. First we recall that, for $0<\delta\ll 1$, a
$\delta$-\textit{tube} is defined to be any rectangular box $T$ in $\mathbb{R}^{d}$
with $d-1$ sides of length $\delta$ and one side of length $1$.
Let
$\mathbb{T}$ be an arbitrary collection of such $\delta$-tubes
whose orientations form a $\delta$-separated set of points on
$\mathbb{S}^{d-1}$.
\begin{conjecture}[Linear Kakeya]\label{LKC}
Let $\varepsilon>0$. If
$\tfrac{1}{q}\leq\tfrac{d-1}{d}$ and $\tfrac{d-1}{p}+\tfrac{1}{q}\leq d-1$,
then there is a constant $C_\varepsilon$, independent
of $\delta$ and the collection $\mathbb{T}$, such that
\begin{equation*}\label{kakeya}
\Bigl\|\sum_{T\in\mathbb{T}}\chi_{T}\Bigr\|_{L^{q}(\mathbb{R}^{d})} \leq
C_\varepsilon\delta^{\frac{d}{q}-\frac{d-1}{p'}-\varepsilon}\left( \#
\mathbb{T}\right)^{\frac{1}{p}}.
\end{equation*}
\end{conjecture}
A simple consequence of this conjecture is the \emph{Kakeya set conjecture}, which asserts that any (Borel) set in
$\mathbb{R}^d$ which contains a unit line segment in every direction must
have full Hausdorff dimension. See the survey articles \cite{wolff:survey} or \cite{TaoBlog}.

Conjecture \ref{LKC} was proved for $d=2$ by C\'ordoba in 1977 \cite{Co}. For details of
the subsequent
progress in higher dimensions see \cite{TaoBlog}.

In the early 1990's Bourgain developed a partial reverse mechanism, showing that progress on the Kakeya conjecture may be used to make progress on the restriction conjecture, and variants of this mechanism (developed by Wolff \cite{Wol} and Tao \cite{TaoKM}) have played a central role in all subsequent progress on the restriction conjecture. This mechanism may be interpreted as a certain inductive argument (or recursive inequality), through which progress on Conjecture \ref{LKC} may be transferred by iteration to progress on Conjecture \ref{LRC}. As this mechanism will feature heavily in Section \ref{MultSect}, we refrain from entering into detail here, and refer the reader to \cite{TaoRestn} for further discussion.



In the 1990s another new perspective was introduced to the restriction problem which also aimed to exploit curvature in a more geometric way: this was the so-called \textit{bilinear approach}.

\section{The bilinear restriction theory: the emergence of transversality}
\subsection{The bilinear restriction conjecture}
A bilinear perspective on the Fourier restriction problem emerged in the 1990's in work of Bourgain, and was later developed systematically by Tao, Vargas and Vega in \cite{TVV}.
A motivating observation was that if $S$ has everywhere nonvanishing gaussian curvature and $f_1,f_2$ are functions on $S$ with \emph{separated supports} then the inequality
\begin{equation}\label{blr}\|\widehat{f_1d\sigma}\widehat{f_2d\sigma}\|_{L^{q/2}(\mathbb{R}^d)}\lesssim\|f_1\|_{L^{p}(d\sigma)}\|f_2\|_{L^{p}(d\sigma)}
\end{equation} typically holds for a much wider range of exponents $p$, $q$ than what is predicted by an application of H\"older's inequality and the restriction conjecture. On one level this so-called ``bilinear improvement" on \eqref{lrc} may be understood through a certain \emph{transversality property inherited from the curvature of the submanifold} $S$; namely if $S$ has everywhere non-vanishing gaussian curvature and $S_1,S_2\subset S$ are separated, then they are generically transversal.

In order to formulate a natural bilinear analogue of Conjecture \ref{LRC}, we let $S_1$ and $S_2$ be compact smooth $(d-1)$-dimensional submanifolds of $\mathbb{R}^d$. We shall assume that $S_1$ and $S_2$ are \emph{transversal} in the sense that if $v_1$ and $v_2$ are unit normal vectors to $S_1$ and $S_2$ respectively, then $|v_1\wedge v_2|$ (the angle between $v_1$ and $v_2$) is bounded below by some constant $c>0$ uniformly in the choices of $v_1$ and $v_2$.\footnote{This notion of transversality is slightly different from the classical notion of transversality from differential geometry; in particular, whether $S_1$ and $S_2$ intersect is not relevant here.} Finally, let $d\sigma_1$ and $d\sigma_2$ denote the induced Lebesgue measure on $S_1$ and $S_2$ respectively.

\begin{conjecture}[Bilinear restriction conjecture]\label{BRC}
Suppose $S_1$ and $S_2$ are transversal with everywhere positive principal curvatures.\footnote{This conjecture is known to be false for surfaces with nonvanishing curvatures of different sign, such as the saddle; see \cite{Lee} and \cite{Vargas}.} If $\frac{1}{q}<\frac{d-1}{d}$, $\frac{d+2}{2q}+\frac{d}{p}\leq d$ and $\frac{d+2}{2q}+\frac{d-2}{p}\leq d-1$ then
\begin{equation}\label{brc}
\|\widehat{f_1d\sigma_1}\widehat{f_2d\sigma_2}\|_{L^{q/2}(\mathbb{R}^d)}\lesssim\|f_1\|_{L^p(d\sigma_1)}\|f_2\|_{L^p(d\sigma_2)}
\end{equation}
for all $f_1\in L^p(d\sigma_1)$ and $f_2\in L^p(d\sigma_2)$.
\end{conjecture}
As with the linear restriction conjecture, the conjectured exponents above may be generated by testing \eqref{brc} on characteristic functions of certain caps. In this case the appropriate caps are eccentric and carefully orientated; see \cite{TVV} for the details.

It is of course natural to question the relative significance of the transversality and curvature hypotheses in Conjecture \ref{BRC}. If we were to drop the transversality hypothesis, the conjectured exponents $p,q$ would simply shrink to those of the linear conjecture. On the other hand, and by contrast with the linear situation, if we drop the curvature hypothesis (retaining the transversality), then it is not difficult to see that nontrivial estimates may be obtained. In particular, in all dimensions we have
\begin{equation}\label{2dbil}
\|\widehat{f_1d\sigma_1}\widehat{f_2d\sigma_2}\|_{L^2(\mathbb{R}^d)}\lesssim\|f_1\|_{L^2(d\sigma_1)}
\|f_2\|_{L^2(d\sigma_2)}.
\end{equation}
Let us see why this is true.
By Plancherel's theorem this estimate is equivalent to $$\|(f_1d\sigma_1)*(f_2d\sigma_2)\|_{L^2(\mathbb{R}^d)}\lesssim\|f_1\|_{L^2(d\sigma_1)}
\|f_2\|_{L^2(d\sigma_2)}.$$
By interpolation with the trivial $\|(f_1d\sigma_1)*(f_2d\sigma_2)\|_{L^1(\mathbb{R}^d)}\lesssim\|f_1\|_{L^1(d\sigma_1)}
\|f_2\|_{L^1(d\sigma_2)}$,
it suffices to prove that $$\|(f_1d\sigma_1)*(f_2d\sigma_2)\|_{L^\infty(\mathbb{R}^d)}\lesssim\|f_1\|_{L^\infty(d\sigma_1)}
\|f_2\|_{L^\infty(d\sigma_2)}.$$
However, $\|(f_1d\sigma_1)*(f_2d\sigma_2)\|_{L^\infty(\mathbb{R}^d)}\leq\|f_1\|_{\infty}\|f_2\|_{\infty}\|d\sigma_1*d\sigma_2\|_{L^\infty(\mathbb{R}^d)}$,
reducing matters to the elementary fact that $d\sigma_1*d\sigma_2\in L^\infty(\mathbb{R}^d)$. We point out that for $d=2$, the inequality \eqref{2dbil} is actually strictly stronger than Conjecture \ref{BRC}, being at the omitted endpoint $(p,q)=(2,4)$ and moreover, holding in the absence of a curvature hypothesis.

As Conjecture \ref{BRC} suggests, if we include both transversality and curvature then for $d\geq 3$ we may improve on \eqref{2dbil} considerably; for example, Tao, Vargas and Vega \cite{TVV}\footnote{Following the $d=3$ case in Moyua, Vargas and Vega \cite{MVV}.} showed that
$$
\|\widehat{f_1d\sigma_1}\widehat{f_2d\sigma_2}\|_{L^2(\mathbb{R}^d)}\lesssim\|f_1\|_{L^{\frac{4d}{3d-2}}(d\sigma_1)}
\|f_2\|_{L^{\frac{4d}{3d-2}}(d\sigma_2)},
$$
and Tao \cite{TaoKM} showed that
\begin{equation}\label{KM}
\|\widehat{f_1d\sigma_1}\widehat{f_2d\sigma_2}\|_{L^{q/2}(\mathbb{R}^d)}\lesssim\|f_1\|_{L^{2}(d\sigma_1)}
\|f_2\|_{L^{2}(d\sigma_2)}
\end{equation}
for all $q>\frac{2(d+2)}{d}$. Tao's proof of \eqref{KM} (which builds on \cite{Wol}) proceeds by a sophisticated variant of Bourgain's inductive method \cite{Bo}, involving a bilinear Kakeya-type ingredient; see \cite{TaoRestn} for further discussion.


\subsection{From bilinear to linear}\label{twothree}
Arguably the most valuable feature of the bilinear restriction conjecture is the fact that it (if formulated in an appropriately scale-invariant way -- see \cite{TVV}) implies the linear restriction conjecture. For technical reasons related to scale-invariance, we confine our attention to the situation where $S$ is a compact subset of a paraboloid.
\begin{prop}[Tao--Vargas--Vega \cite{TVV}]\label{bilinimplieslin}
Suppose that $S$ is a compact subset of a paraboloid and that $S_1$ and $S_2$ are transversal subsets of $S$. If $\tfrac{1}{q}<\tfrac{d-1}{2d}$, $\tfrac{1}{q}
\leq\tfrac{d-1}{d+1}\tfrac{1}{p'}$ and the conjectured bilinear inequality
\begin{equation*}
\|\widehat{f_1d\sigma}\widehat{f_2d\sigma}\|_{L^{\widetilde{q}/2}(\mathbb{R}^d)}\lesssim\|f_1\|_{L^{\widetilde{p}}(S_1)}\|f_2\|_{L^{\widetilde{p}}(S_2)}
\end{equation*}
holds for all $(\widetilde{p}, \widetilde{q})$ in a neighbourhood of $(p, q)$ then the conjectured linear inequality
\begin{equation*}
\|\widehat{fd\sigma}\|_{L^{q}(\mathbb{R}^d)}\lesssim\|f\|_{L^{p}(d\sigma)}
\end{equation*}
holds.
\end{prop}
This bilinear approach to the linear restriction conjecture has been very successful. Until very recently the state-of-the-art on Conjecture \ref{LRC} relied upon this passage and Tao's bilinear inequality \eqref{KM}.

In order to indicate why Proposition \ref{bilinimplieslin} is true it is natural that we present an argument that may be adapted to a more general multilinear setting. While the rather simple argument that we sketch here -- due to Bourgain and Guth \cite{BG} -- has this advantage, it does not appear to easily capture the full strength of the proposition. Here we will indicate how the conjectured bilinear inequality \eqref{brc} may be used to obtain the conjectured linear inequality \eqref{lrc} whenever $p=q>\frac{2d}{d-1}$. This special case is readily seen to imply the linear restriction conjecture on the interior of the full conjectured range of Lebesgue exponents.

Let $\{S_\alpha\}$ be a partition of $S$ by patches of diameter approximately $1/K$ and write $$f=\sum_\alpha f_\alpha,\;\;\mbox{ where }\;\;f_\alpha=f\chi_{S_\alpha}.$$
By linearity
$$\widehat{fd\sigma}=\sum_\alpha\widehat{f_\alpha d\sigma}.$$ The key observation is the following elementary inequality; see \cite{BG}.
\begin{prop}\label{bilBG}
\begin{equation}\label{BGeasy}
|\widehat{fd\sigma}(\xi)|^q\lesssim K^{2(d-1)q}\sum_{S_{\alpha_1},S_{\alpha_2} }|\widehat{f_{\alpha_1}d\sigma}(\xi)\widehat{f_{\alpha_2}d\sigma}(\xi)|^{\frac{q}{2}}+\sum_{\alpha}|\widehat{f_{\alpha}d\sigma}(\xi)|^q,
\end{equation}
where the sum in $S_{\alpha_1},S_{\alpha_2}$ is restricted to $1/K$-transversal pairs $S_{\alpha_1}, S_{\alpha_2}$.
\end{prop}
By $1/K$-transversal we mean that $|v_1\wedge v_2|\geq 1/K$ for all choices of unit normal vectors $v_1,v_2$ to $S_1,S_2$ respectively.
This proposition essentially amounts to an application of the elementary abstract inequality
$$\|a\|_{\ell^1(\mathbb{Z}_N)}^q\lesssim N^{2q}\sum_{j\not=k}|a_ja_k|^{\frac{q}{2}}+\|a\|_{\ell^q(\mathbb{Z}_N)}^q$$
for finite sequences of real numbers $a$. We leave the verification of this to the interested reader.

Assuming the truth of Proposition \eqref{bilBG}, and integrating in $\xi$, we obtain
\begin{equation}\label{BGint}\|\widehat{fd\sigma}\|_q^q\lesssim K^{2(d-1)q}\sum_{S_{\alpha_1},S_{\alpha_2} }\|\widehat{f_{\alpha_1}d\sigma}\widehat{f_{\alpha_2}d\sigma}\|_{q/2}^{\frac{q}{2}}+\sum_{\alpha}\|\widehat{f_{\alpha}d\sigma}\|_q^q,
\end{equation}
which, because of the terms $\|\widehat{f_{\alpha}d\sigma}\|_q^q$ appearing on the right-hand-side, strongly suggests the viability of a bootstrapping argument. To this end let $\mathcal{C}=\mathcal{C}(R)$ denote the smallest constant in the inequality $\|\widehat{fd\sigma}\|_{L^q(B(0,R))}\leq C\|f\|_q$ over all $R\gg 1$ and $f\in L^p(d\sigma)$.
The only role of the parameter $R$ here is to ensure that $\mathcal{C}$ is a-priori finite, and of course the task is to show that $\mathcal{C}<\infty$ uniformly in $R$.
Using a parabolic scaling\footnote{It is here where technical difficulties arise when $S$ is other than a piece of paraboloid. To overcome this, it is necessary to include uniformity ingredients in both the definition of $\mathcal{C}$ and Conjecture \ref{BRC}; see \cite{BCT} and \cite{BG} for details.}  and the definition of $\mathcal{C}$ allows us to deduce that
$$\|\widehat{f_\alpha d\sigma}\|_q\lesssim\mathcal{C}K^{2d/q-(d-1)}\|f_\alpha\|_q,$$ which represents a gain for large $K$ as the power $2d/q-(d-1)$ is negative. Using \eqref{BGint} along with the support disjointness property $\sum_\alpha\|f_\alpha\|_q^q=\|f\|_q^q$, we obtain
\begin{equation}\label{BGint2}\|\widehat{fd\sigma}\|_q^q\leq c K^{2(d-1)q}\sum_{S_{\alpha_1},S_{\alpha_2} }\|\widehat{f_{\alpha_1}d\sigma}\widehat{f_{\alpha_2}d\sigma}\|_{q/2}^{\frac{q}{2}}+c\mathcal{C}K^{2d/q-(d-1)}\|f\|_q^q
\end{equation}
for some constant $c$ independent of $K$. Now, taking $K$ sufficiently large so that $cK^{2d/q-(d-1)}\leq 1/2$ (say), we see that it suffices to show that
\begin{equation}\label{dagger}
K^{2(d-1)q}\sum_{S_{\alpha_1},S_{\alpha_2} }\|\widehat{f_{\alpha_1}d\sigma}\widehat{f_{\alpha_2}d\sigma}\|_{q/2}^{\frac{q}{2}}\leq A\|g\|_q^q
\end{equation}
for some constant $A=A(K)$. Indeed, if we have \eqref{dagger} then by the definition of $\mathcal{C}$ we have $\mathcal{C}\leq cA+\mathcal{C}/2$, from which we may deduce that $\mathcal{C}<\infty$ uniformly in $R$. However, \eqref{dagger} is a straightforward consequence of the conjectured bilinear inequality \eqref{brc}.

It is perhaps helpful to remark that the above argument would have been equally effective if the factor of $K^{2(d-1)}$ in \eqref{BGeasy} were replaced by any fixed power of $K$. As we have seen, the key feature of \eqref{BGeasy} is the absence of a power of $K$ in the second (``bootstrapping") term on the right-hand-side.

While the bilinear approach to the restriction conjecture has proved very powerful, it does have some drawbacks for $d\geq 3$. In particular,
\begin{itemize}
\item[(i)] it leaves some confusion over the relative roles of curvature and transversality, and
\item[(ii)] it provides very limited insight into the relationship between the restriction and Kakeya problems.
\end{itemize}
As we shall see in the next section, in a multilinear setting these matters are in some sense clarified.
\section{Multilinear transversality and the multilinear restriction theory}\label{MultSect}
\subsection{The multilinear restriction problem}
The notion of transversality discussed previously has a very natural extension to the context of several codimension-one submanifolds of $\mathbb{R}^d$ provided the dimension $d$ is large enough.

\begin{definition}[multilinear transversality]
Let $2\leq k\leq d$ and $c>0$. A $k$-tuple $S_1,\hdots,S_k$ of smooth codimension-one submanifolds of $\mathbb{R}^d$ is $c$-\emph{transversal} if $$|v_{1}\wedge\cdots\wedge v_{k}|\geq c$$ for all choices $v_{1},\hdots, v_{k}$ of unit normal vectors to $S_{1},\hdots, S_{k}$ respectively. We say that $S_1,\hdots, S_k$ are transversal if they are $c$-transversal for some $c>0$.
\end{definition}
In the above definition the $k$-dimensional volume form $|v_{1}\wedge\cdots\wedge v_{k}|$ is simply the $k$-dimensional volume of the parallelepiped generated by $v_1,\hdots, v_k$.

\begin{conjecture}[$k$-linear Restriction]\label{MLRC}
Let $k\geq 2$ and suppose $S_1,\hdots,S_k$ are transversal with everywhere positive principal curvatures. If $\tfrac{1}{q}<\tfrac{d-1}{2d}$,
$\tfrac{1}{q}\leq\tfrac{d+k-2}{d+k}\tfrac{1}{p'}$ and
$\tfrac{1}{q}\leq\tfrac{d-k}{d+k}\tfrac{1}{p'}+\tfrac{k-1}{k+d}$,
then
\begin{equation}\label{mri}
\Bigl\|\prod_{j=1}^{k}\widehat{f_jd\sigma_j}\Bigr\|_{L^{q/k}(\mathbb{R}^d)}\lesssim\prod_{j=1}^{k} \|f_{j}\|_{L^{p}(d\sigma_{j})}
\end{equation}
for all $f_{1}\in L^p(d\sigma_1),\hdots,f_{k}\in L^{p}(d\sigma_k)$.
\end{conjecture}
As may be expected, the case $k=d$ is rather special.
At this level of multilinearity Conjecture \ref{MLRC} changes nature quite fundamentally, exhibiting a number of important features which are not seen at lower levels of multilinearity. Indeed when $k=d$, standard examples indicate that Conjecture \ref{MLRC} may be strengthened as follows.
\begin{conjecture}[$d$-linear Restriction Conjecture]\label{dLRC}
If $S_1,\hdots,S_d$ are transversal,
$\tfrac{1}{q}\leq\tfrac{d-1}{2d}$ \emph{and} $\tfrac{1}{q}
\leq\tfrac{d-1}{d}\tfrac{1}{p'}$, then
\begin{equation*}
\|\prod_{j=1}^d\widehat{f_jd\sigma_j}\|_{L^{q/d}(\mathbb{R}^d)}\lesssim\prod_{j=1}^d\|f_j\|_{L^{p}(d\sigma_j)}.
\end{equation*}
\end{conjecture}
Of course the most striking difference is that the curvature hypothesis has been completely removed. Furthermore, by multilinear interpolation it is easy to see that Conjecture \ref{dLRC} is equivalent to an endpoint inequality
\begin{equation}\label{endp}
\|\prod_{j=1}^d\widehat{f_jd\sigma_j}\|_{L^{\frac{2}{d-1}}(\mathbb{R}^d)}\lesssim\prod_{j=1}^d\|f_j\|_{L^{2}(d\sigma_j)},
\end{equation}
which was previously excluded. We remark in passing that \eqref{endp}, being on $L^2$, has a rather elegant interpretation as a multilinear example of the Strichartz estimates discussed in Section \ref{twoone}. In this context it is a little more convenient to first rephrase \eqref{endp} in terms of parametrised surfaces. If $\Sigma_1:U_1\rightarrow\mathbb{R}^d, \cdots, \Sigma_d:U_d\rightarrow\mathbb{R}^d$ are smooth parametrisations of $S_1,\hdots, S_d$ with associated extension operators $\mathcal{E}_1,\hdots,\mathcal{E}_d$, respectively, then the conjectured endpoint inequality
\eqref{endp} may be restated as
\begin{equation}\label{endpparr}
\|\prod_{j=1}^d\mathcal{E}_jg_j\|_{L^{\frac{2}{d-1}}(\mathbb{R}^d)}\lesssim\prod_{j=1}^d\|g_j\|_{L^2(U_j)}.
\end{equation}
Now, for example, suppose that $u_1,\hdots,u_d:\mathbb{R}\times\mathbb{R}^d\rightarrow\mathbb{C}$ are solutions to the Schr\"odinger equation $i\partial_tu+\Delta u=0$. As is straightforward to verify, the transversality hypothesis requires that the supports of the spatial Fourier transforms of the initial data $\supp(\widehat{u}_1(0,\cdot)),\hdots, \supp(\widehat{u}_d(0,\cdot))$ are compact and meet no affine hyperplane in $\mathbb{R}^{d-1}$. An application of Plancherel's theorem reveals that under this condition \eqref{endpparr} becomes
$$
\|\prod_{j=1}^k u_j\|_{L^{\frac{2}{d-1}}_{t,x}(\mathbb{R}\times\mathbb{R}^{d-1})}\lesssim \prod_{j=1}^d\|u_j(0,\cdot)\|_{L^{2}(\mathbb{R}^{d-1})}.
$$

Unlike at lower levels of multilinearity, the $d$-linear restriction conjecture is almost resolved.
\begin{theorem}[B--Carbery--Tao \cite{BCT}]\label{bcttheorem}
If $S_1,\hdots,S_d$ are transversal then given any $\varepsilon>0$ there exists a constant $C_\varepsilon<\infty$ such that
$$
\|\prod_{j=1}^d\widehat{f_jd\sigma_j}\|_{L^{\frac{2}{d-1}}(B(0,R))}\leq C_\varepsilon R^\varepsilon\prod_{j=1}^d\|f_j\|_{L^{2}(d\sigma_j)}
$$
for all $R$.
\end{theorem}


As we shall see in Section \ref{fourthree}, the recent work of Guth \cite{G}, combined with a slight reworking of certain elements of \cite{BCT}, leads to the following modest improvement on the above result:
\begin{theorem}\label{bcttheoremrevisted}
If $S_1,\hdots,S_d$ are transversal then there exist constants $C$ and $\kappa$ such that
$$\|\prod_{j=1}^d\widehat{f_jd\sigma_j}\|_{L^{\frac{2}{d-1}}(B(0,R))}\leq C(\log R)^\kappa \prod_{j=1}^d\|f_j\|_{L^2(d\sigma_j)}$$
for all $R>0$.
\end{theorem}

Another special feature of the $d$-linear restriction problem is its relation with certain multilinear geometric inequalities. Exploring this turns out to be particularly revealing, and is a natural precursor to any discussion of multilinear Kakeya-type problems.
\subsection{Relation with multilinear geometric inequalities}\label{twotwo} 
Here it will be a little more convenient to work with parametrised extension operators \eqref{extn}, where the conjectured endpoint $d$-linear restriction inequality becomes \eqref{endpparr}.
Since there are no curvature hypotheses in the $d$-linear restriction conjecture it is natural to look at the conjectured endpoint $d$-linear restriction inequality in the situation where $S_j$ is the $j$th coordinate subspace $\{x=(x_1,\hdots,x_d):x_j=0\}$.
In this case
$\ex_jg_j=\widehat{g}_j\circ\pi_j$ where $\widehat{\;}$ denotes the $(d-1)$-dimensional Fourier transform and $\pi_j:\mathbb{R}^d\rightarrow\mathbb{R}^{d-1}$ is given by $$\pi_j(\xi)
=(\xi_1,\hdots,\xi_{j-1},\xi_{j+1},\hdots,\xi_d).$$
Thus the conjectured endpoint inequality \eqref{endpparr} becomes
$$\|\widehat{g}_1\circ\pi_1\cdots\widehat{g}_d\circ\pi_d\|_{L^{\frac{2}{d-1}}(\mathbb{R}^d)}\lesssim\|g_1\|_{L^{2}(S_1)}\cdots\|g_d\|_{L^{2}(S_d)}.$$
This, by Plancherel's theorem, reduces to
$$\|g_1\circ\pi_1\cdots g_d\circ\pi_d\|_{L^{\frac{2}{d-1}}(\mathbb{R}^d)}\lesssim\|g_1\|_{L^{2}(S_1)}\cdots\|g_d\|_{L^{2}(S_d)},$$
which on setting $f_j=|g_j|^2$ is equivalent to the (positive) inequality
$$
\int_{\mathbb{R}^d}(f_1\circ\pi_1)^{\frac{1}{d-1}}\cdots (f_d\circ\pi_d)^{\frac{1}{d-1}}\lesssim\Bigl(\int_{\mathbb{R}^{d-1}}f_1\Bigr)^{\frac{1}{d-1}}\cdots\Bigl(\int_{\mathbb{R}^{d-1}}f_d\Bigr)^{\frac{1}{d-1}}.$$
This is the \emph{Loomis--Whitney inequality} with a suboptimal constant.
\begin{theorem}[Loomis--Whitney 1948]
$$
\int_{\mathbb{R}^d}(f_1\circ\pi_1)^{\frac{1}{d-1}}\cdots (f_d\circ\pi_d)^{\frac{1}{d-1}}\leq\Bigl(\int_{\mathbb{R}^{d-1}}f_1\Bigr)^{\frac{1}{d-1}}\cdots\Bigl(\int_{\mathbb{R}^{d-1}}f_d\Bigr)^{\frac{1}{d-1}}$$
for all nonnegative integrable functions $f_1,\hdots,f_d$ on $\mathbb{R}^{d-1}$.
\end{theorem}

As Loomis and Whitney point out in \cite{LW}, this inequality is a geometric inequality closely related to the classical isoperimetric inequality. To see this suppose that $\Omega\subset\mathbb{R}^d$ has finite measure. Setting $f_j=\chi_{\pi_j(\Omega)}$ we have that $f_j\circ\pi_j(x)=1$ whenever $\;x\in\Omega$, and so
by the Loomis--Whitney inequality,
\begin{equation}\label{lwg}
|\Omega|\leq |\pi_1(\Omega)|^{\frac{1}{d-1}}\cdots |\pi_d(\Omega)|^{\frac{1}{d-1}}.
\end{equation}
Now, since $|\pi_j(\Omega)|\leq |\partial \Omega|$ for each $j$, we recover the classical isoperimetric inequality $$|\Omega|\leq |\partial\Omega|^{\frac{d}{d-1}},$$ (albeit with suboptimal constant).
Notice also that \eqref{lwg} tells us that
$$|\Omega|\geq \frac{|\Omega|}{|\pi_1(\Omega)|}\cdots\frac{|\Omega|}{|\pi_d(\Omega)|};$$ that is, the Lebesgue measure of a subset of $\mathbb{R}^d$ is bounded below by the product of its ``average widths" in $d$ orthogonal directions.

We remark in passing that the Loomis--Whitney has a rather useful affine-invariant formulation, whereby the particular mappings $\pi_j$ are replaced by general surjections $L_j:\mathbb{R}^d\rightarrow\mathbb{R}^{d-1}$ with the property that $\{\ker L_1,\hdots,\ker L_d\}$ forms a basis of $\mathbb{R}^d$. This follows from the standard case after suitable linear changes of variables.

The simplest proof of the Loomis--Whitney inequality consists of a repeated use of the multilinear H\"older inequality.
For example, for $d=3$ we may use the Cauchy--Schwarz inequality twice to write
\begin{eqnarray*}
\begin{aligned}
\int_{\mathbb{R}^3}f_1(x_2,x_3)^{\frac{1}{2}}&f_2(x_1,x_3)^{\frac{1}{2}}f_3(x_1,x_2)^{\frac{1}{2}}dx\\&=\int_{\mathbb{R}^2}\Bigl(\int_{\mathbb{R}}f_1(x_2,x_3)^{\frac{1}{2}}f_2(x_1,x_3)^{\frac{1}{2}}dx_3\Bigr)
f_3(x_1,x_2)^{\frac{1}{2}}dx'\\
&\leq\int_{\mathbb{R}^2}\Bigl(\int_{\mathbb{R}}f_1(x_2,\cdot)\Bigr)^{\frac{1}{2}}\Bigl(\int_{\mathbb{R}}f_2(x_1,\cdot)\Bigr)^{\frac{1}{2}}f_3(x_1,x_2)^{\frac{1}{2}}dx'\\
&\leq\Bigl(\int_{\mathbb{R}^2}f_1\Bigr)^{\frac{1}{2}}\Bigl(\int_{\mathbb{R}^2}f_2\Bigr)^{\frac{1}{2}}\Bigl(\int_{\mathbb{R}^2}f_3\Bigr)^{\frac{1}{2}}.
\end{aligned}
\end{eqnarray*}
As should be expected, this proof of this special case of the $d$-linear restriction conjecture (where the submanifolds $S_j$ are transversal hyperplanes) does not extend routinely to general transversal $S_1,\hdots,S_d$.  However, an important aspect of it does: if one is prepared to lose an $\varepsilon$ or a logarithmic factor as in Theorems \ref{bcttheorem} or \ref{bcttheoremrevisted}, then one may indeed reduce the general case to a \textit{positive} inequality of Loomis--Whitney type. This positive inequality is the $d$-linear Kakeya inequality, which we now describe.

\subsection{Relation with the multilinear Kakeya problem}\label{fourthree}
As in the linear setting, the multilinear restriction conjecture (Conjecture \ref{MLRC}) implies a certain multilinear Kakeya-type conjecture
involving $2\leq k\leq d$ ``transversal" families of $\delta$-tubes
$\mathbb{T}_1,\hdots,\mathbb{T}_k$ in $\mathbb{R}^d$. The appropriate notion of transversality here is inherited from that for codimension-$1$ submanifolds: we say that
the families $\tubes_1,\hdots,\tubes_k$ of tubes in $\mathbb{R}^d$ are \emph{transversal} if there is a constant $c>0$ such that given any collection of tubes $T_1\in\tubes_1,\hdots,T_k\in\tubes_k$, \begin{equation}\label{naivetrans}|e(T_1)\wedge\cdots\wedge e(T_k)|\geq c.
\end{equation}
Here $e(T)\in\mathbb{S}^{d-1}$ denotes the direction of the long side of a tube $T$.
\begin{conjecture}[Multilinear Kakeya]\label{MLKC}
Let $\varepsilon>0$ and $d\geq k\geq 2$. Suppose that $\tubes_{1},\hdots,\tubes_k$ are transversal families of $\delta$-tubes such that for each $1\leq j\leq k$, $\{e(T_j):T_j\in\mathbb{T}_j\}$ forms a $\delta$-separated subset of $\mathbb{S}^{d-1}$. If $\tfrac{1}{q}\leq \tfrac{d-1}{d}$ and $
\tfrac{d-k}{p}+\tfrac{k}{q}\leq d-1,$ then there is a constant $C_\varepsilon$, independent
of $\delta$ and the families $\tubes_{1},\hdots,\tubes_k$, such that
\begin{equation}\label{mk}
\Bigl\|\prod_{j=1}^{k}\Bigl(\sum_{T_{j}\in\tubes_{j}}\chi_{T_{j}}\Bigr)\Bigr\|_{L^{q/k}(\mathbb{R}^d)}\leq
C_\varepsilon\prod_{j=1}^{k}\delta^{\frac{d}{q}-\frac{d-1}{p'}-\varepsilon}\:
(\#\tubes_{j})^{ \frac{1}{p}}.
\end{equation} 
\end{conjecture}
This conjecture is elementary for $d=2$. For details of the
progress for $d\geq 3$,
see \cite{TVV} and
\cite{BCT}.

As with the multilinear restriction conjecture, the extreme case $k=d$ is rather special, and indeed much more can be said. In particular, the angular $\delta$-separation condition (which, as discussed in Section \ref{two}, is the manifestation of curvature in this setting) within each family of tubes $\mathbb{T}_j$ may be dropped, as can the $\varepsilon$ loss in \eqref{mk}. Moreover, the tubes themselves may have arbitrary (possibly infinite) length.
Unusually for euclidean Kakeya-type problems, this case has been resolved completely.
\begin{theorem}[$d$-linear Kakeya \cite{BCT}, \cite{G}]\label{guthkak}
Let $\mathbb{T}_1,\hdots, \mathbb{T}_d$ be families of doubly-infinite $\delta$-tubes. If these families are transversal and $q\geq \frac{d}{d-1}$ then
\begin{equation*}
\Bigl\|\prod_{j=1}^d\Bigl(\sum_{T_j\in\mathbb{T}_j}\chi_{T_j}\Bigr)\Bigr\|_{L^{q/d}(\mathbb{R}^d)}\lesssim\prod_{j=1}^d\delta^{d/q}\#\mathbb{T}_j.\end{equation*}
\end{theorem}
Theorem \ref{guthkak} is due to Carbery, Tao and the author for $q>\frac{d}{d-1}$, and Guth at the endpoint $q=\frac{d}{d-1}$.

Some remarks are in order. Theorem \ref{guthkak} is equivalent to Guth's endpoint estimate
\begin{equation}\label{guthest}
\Bigl\|\prod_{j=1}^d\Bigl(\sum_{T_j\in\mathbb{T}_j}\chi_{T_j}\Bigr)\Bigr\|_{L^{\frac{1}{d-1}}(\mathbb{R}^d)}\lesssim\prod_{j=1}^d\delta^{d-1}\#\mathbb{T}_j,\end{equation} which is a certain generalisation of the Loomis--Whitney inequality in disguise. Indeed, \eqref{guthest} has an equivalent functional form which may be viewed as a certain ``vector" or ``combinatorial" Loomis--Whitney inequality, namely
\begin{equation}\label{comblw}
\int_{\mathbb{R}^d}\prod_{j=1}^d\Bigl(\sum_{\alpha_j\in\mathcal{A}_j}f_{\alpha_j}\circ\pi_{\alpha_j}\Bigr)^{\frac{1}{d-1}}\lesssim\prod_{j=1}^d\Bigl(\sum_{\alpha_j\in\mathcal{A}_j}\int_{\mathbb{R}^{d-1}}
f_{\alpha_j}\Bigr)^{\frac{1}{d-1}},
\end{equation}
where for each $j$, $\mathcal{A}_j$ is an indexing set and $\pi_{\alpha_j}$ is a linear map which is sufficiently close to the fixed $\pi_j$ (the $j$th coordinate hyperplane projection). It should be remarked that the $k$-linear Kakeya conjecture only has such an equivalent functional form when $k=d$, a feature which relies crucially on the absence of the angular $\delta$-separation (or curvature) condition.

To see that \eqref{comblw} implies \eqref{guthest} we simply set $f_{\alpha_j}=\chi_{B(\alpha_j)}$, where $B(\alpha_j)$ denotes a $\delta$-ball in $\mathbb{R}^{d-1}$, then $$f_{\alpha_j}\circ\pi_{\alpha_j}=\chi_{T(\alpha_j)},$$ where $T(\alpha_j)=\pi_{\alpha_j}^{-1}B(\alpha_j)$ is a doubly infinite cylindrical tube in $\mathbb{R}^d$ of width $\sim\delta$ and direction $\ker\pi_{\alpha_j}$. To see that \eqref{guthest} implies \eqref{comblw} we observe that it suffices, by scaling and a density argument, to prove \eqref{comblw} for input functions $f_{\alpha_j}$ which are finite sums of characteristic functions of $\delta$-balls in $\mathbb{R}^{d-1}$.
Similar considerations reveal that \eqref{guthest} also self-improves to
\begin{equation}\label{guthest3}
\Bigl\|\prod_{j=1}^d\Bigl(\sum_{T_j\in\mathbb{T}_j}\mu_{T_j}*\chi_{T_j}\Bigr)\Bigr\|_{L^{\frac{1}{d-1}}(\mathbb{R}^d)}^{\frac{1}{d-1}} \lesssim\delta^d\Bigl(\prod_{j=1}^d\|\mu_{T_j}\|\Bigr)^{\frac{1}{d-1}},
\end{equation}
where $\mu_{T_j}$ is a finite measure on $\mathbb{R}^d$ for each $T_j\in\mathbb{T}_j$, $1\leq j\leq d$. As such calculations show, the implicit constant here is the same as that in \eqref{guthest}.

Perhaps the most important feature of the $d$-linear Kakeya inequality \eqref{guthest} is that it possess a certain scale-invariance, or ``self-similarity", property that we now describe.
For each $0<\delta\ll 1$ let $\mathcal{C}_{\kak}(\delta)$ denote the smallest constant $C$ in
\begin{equation}\label{guthest2}
\Bigl\|\prod_{j=1}^d\Bigl(\sum_{T_j\in\mathbb{T}_j}\chi_{T_j}\Bigr)\Bigr\|_{L^{\frac{1}{d-1}}(\mathbb{R}^d)}^{\frac{1}{d-1}}\leq C\delta^d\Bigl(\prod_{j=1}^d\#\mathbb{T}_j\Bigr)^{\frac{1}{d-1}}
\end{equation}
over all transversal families $\mathbb{T}_1,\hdots,\mathbb{T}_d$ of $\delta\times\cdots\times\delta\times 1$-tubes.
(The $d$-linear Kakeya theorem tells us that $\mathcal{C}_\kak (\delta)\lesssim 1$.)

The following proposition is somewhat implicit in \cite{BCT}.
\begin{prop}\label{kakkakkak}
There exists a constant $c\geq 1$ independent of $0<\delta\leq\delta'\leq 1$ such that
\begin{equation}\label{kakrecursive}
\mathcal{C}_{\kak}(\delta)\leq c\mathcal{C}_{\kak}(\delta/\delta')\mathcal{C}_{\kak}(\delta').
\end{equation}
\end{prop}
\begin{proof}
Let $0<\delta\leq\delta'\leq 1$ and tile $\mathbb{R}^d$ by cubes $Q$ of diameter $\delta'$.
Clearly
\begin{eqnarray}\label{kakdec}
\begin{aligned}
\Bigl\|\prod_{j=1}^d\Bigl(\sum_{T_j\in\mathbb{T}_j}\chi_{T_j}\Bigr)\Bigr\|_{L^{\frac{1}{d-1}}(\mathbb{R}^d)}^{\frac{1}{d-1}}&=
\sum_Q\Bigl\|\prod_{j=1}^d\Bigl(\sum_{T_j\in\mathbb{T}_j}\chi_{T_j}\Bigr)\Bigr\|_{L^{\frac{1}{d-1}}(Q)}^{\frac{1}{d-1}}\\
&=\sum_Q\Bigl\|\prod_{j=1}^d\Bigl(\sum_{T_j\in\mathbb{T}_j^Q}\chi_{T_j\cap Q}\Bigr)\Bigr\|_{L^{\frac{1}{d-1}}(\mathbb{R}^d)}^{\frac{1}{d-1}},
\end{aligned}
\end{eqnarray}
where $$\mathbb{T}_j^Q=\{T_j\in\mathbb{T}_j: T_j\cap Q\not=0\}.$$
Fix a cube $Q$. Observe that $T_j\cap Q$ is contained in a tube of long side $\delta'$ and $d-1$ short sides of length $\delta$. Applying the definition of $\mathcal{C}_\kak$, after a suitable scaling, reveals that
\begin{equation}\label{kakdec1}
\Bigl\|\prod_{j=1}^d\Bigl(\sum_{T_j\in\mathbb{T}_j^Q}\chi_{T_j\cap Q}\Bigr)\Bigr\|_{L^{\frac{1}{d-1}}(\mathbb{R}^d)}^{\frac{1}{d-1}}\leq \delta^d\mathcal{C}_\kak(\delta/\delta')\Bigl(\prod_{j=1}^d\#\mathbb{T}_j^Q\Bigr)^{\frac{1}{d-1}}.
\end{equation}
Now, for each $T_j$ let $\widetilde{T}_j=T_j+B(0,c_1\delta')$, where $c_1\geq 1$ is chosen so that if $T_j\cap Q\not=\emptyset$ then $Q\subseteq \widetilde{T}_j$. Thus $\widetilde{T}_j$ is contained in an $O(\delta)\times\cdots\times O(\delta)\times O(1)$-tube, with the same centre and orientation as $T_j$. With this definition of $\widetilde{T}_j$ we have
$$
\#\mathbb{T}_j^Q\leq\sum_{T_j\in\mathbb{T}_j}\chi_{\widetilde{T}_j}(x_Q)$$
for all points $x_Q\in Q$, and so
\begin{equation*}
\sum_Q\Bigl(\prod_{j=1}^d\#\mathbb{T}_j^Q\Bigr)^{\frac{1}{d-1}}\leq\sum_Q\Bigl(\prod_{j=1}^d\sum_{T_j\in\mathbb{T}_j}\chi_{\widetilde{T}_j}(x_Q)\Bigr)^{\frac{1}{d-1}},
\end{equation*}
which upon averaging in all possible choices of points $x_Q\in Q$ for each $Q$ yields
\begin{equation}\label{kakdec3}
\sum_Q\Bigl(\prod_{j=1}^d\#\mathbb{T}_j^Q\Bigr)^{\frac{1}{d-1}}\lesssim (\delta')^{-d}\int_{\mathbb{R}^d}\Bigl(\prod_{j=1}^d\sum_{T_j\in\mathbb{T}_j}\chi_{\widetilde{T}_j}\Bigr)^{\frac{1}{d-1}}.
\end{equation}
Applying the definition of $\mathcal{C}_\kak$ once again\footnote{There is a minor technical point here. The ``tubes" $\widetilde{T}_j$ are a little larger than $\delta'\times\cdots\times\delta'\times 1$-tubes, as the definition of $\mathcal{C}_\kak(\delta')$ requires. However, each $\widetilde{T}_j$ may be covered by boundedly many of these admissible tubes, properly justifying the bound \eqref{kakdec4}.} gives
\begin{equation}\label{kakdec4}
\int_{\mathbb{R}^d}\Bigl(\prod_{j=1}^d\sum_{T_j\in\mathbb{T}_j}\chi_{\widetilde{T}_j}\Bigr)^{\frac{1}{d-1}}\lesssim (\delta')^d\mathcal{C}_\kak(\delta')\Bigl(\prod_{j=1}^d\#\mathbb{T}_j\Bigr)^{\frac{1}{d-1}}.
\end{equation}
Combining \eqref{kakdec}, \eqref{kakdec1}, \eqref{kakdec3} and \eqref{kakdec4} gives $\mathcal{C}_\kak(\delta)\lesssim\mathcal{C}_\kak(\delta/\delta')\mathcal{C}_\kak(\delta')$, with implicit constant uniform in $0<\delta\leq\delta'\leq 1$.
\end{proof}

Proposition \ref{kakkakkak} served as important motivation for the heat-flow approach to the $d$-linear Kakeya conjecture in \cite{BCT}, that proved to be effective away from the endpoint. Indeed one may interpret the monotonicity formulae in \cite{BCT} as recursive inequalities with respect to the continuous time parameter (see also \cite{Bennett} for a broader discussion of this perspective).
As we shall see next, a closely-related argument leads to a similar recursive inequality which explicitly connects the $d$-linear restriction and Kakeya inequalities. This argument, which originates in Bourgain \cite{Bo} in a linear setting, leads to the somewhat surprising \emph{near equivalence} of the $d$-linear restriction and Kakeya problems.

Fix $S_1,\hdots, S_d$ transversal, and for each $R\geq 1$ let $\mathcal{C}_{\rest}(R)$ denote the smallest constant in the inequality
$$\Bigl\|\prod_{j=1}^d\widehat{f_jd\sigma_j}\Bigr\|_{L^{\frac{2}{d-1}}(B(0,R))}\leq C\prod_{j=1}^d\|f_j\|_{L^2(d\sigma_j)}.$$
The endpoint $d$-linear restriction conjecture is thus $\mathcal{C}_\rest(R)\lesssim 1$, and Theorem \ref{bcttheoremrevisted} states that $\mathcal{C}_\rest(R)\lesssim (\log R)^\kappa$ for some $\kappa>0$.
For technical reasons it will be convenient to formulate our recursive inequality in terms of a minor variant of $\mathcal{C}_\rest(R)$. Let $\widetilde{\mathcal{C}}_\rest(R)$ denote the smallest constant $C$ for which the inequality
\begin{equation}\label{enoughconst}
\Bigl\|\prod_{j=1}^d \widehat{f}_j\Bigr\|_{L^{\frac{2}{d-1}}(B(0;R))}\leq CR^{-\frac{d}{2}}\prod_{j=1}^d\|f_j\|_2
\end{equation}
holds
over all $R\geq 1$ and all functions $f_1,\hdots,f_d$ with $\supp (f_j)\subseteq A_j(R):=S_j+B(0,c/R)$ for each $j$. Here $c$ is a positive constant which will be taken sufficiently large for certain technical matters to simplify.
The quantities $\mathcal{C}_\rest(R)$ and $\widetilde{\mathcal{C}}_\rest(R)$
are connected by the following elementary manifestation of the uncertainty principle, whose proof is a straightforward adaptation of that of Proposition 4.3 in \cite{TVV}.
\begin{lemma}\label{tvvlemma}
$\mathcal{C}_\rest(R)\lesssim\widetilde{\mathcal{C}}_\rest(R)$ with implicit constant independent of $R$.
\end{lemma}
The following key proposition is somewhat implicit in \cite{BCT}.
\begin{prop}\label{KRimpliesR} There exists a constant $c\geq 1$ independent of $R$ such that
\begin{equation}\label{restrec}\widetilde{\mathcal{C}}_{\rest}(R)\leq c\widetilde{\mathcal{C}}_{\rest}(R^{1/2})
\mathcal{C}_{\kak}(R^{-1/2}).\end{equation}
\end{prop}
Before we prove this proposition, let us see how it may be combined with Theorem \ref{guthkak} in order to deduce Theorem \ref{bcttheoremrevisted}. Since Guth's endpoint $d$-linear Kakeya inequality amounts to the bound $\mathcal{C}_{\kak}(\delta)\leq C$ for some constant $C$ independent of $\delta$, by Proposition \ref{KRimpliesR} applied $O(\log\log R)$ times we have
\begin{equation}\label{iterrest}
\widetilde{\mathcal{C}}_{\rest}(R)\leq cC\widetilde{\mathcal{C}}_{\rest}(R^{1/2})\leq (cC)^2\widetilde{\mathcal{C}}_{\rest}(R^{1/4})\leq\cdots\leq (cC)^{O(\log\log R)}\widetilde{\mathcal{C}}_\rest(100).
\end{equation}
Theorem \ref{bcttheoremrevisted} now follows from Lemma \ref{tvvlemma} on observing that $(cC)^{O(\log\log R)}=O(\log R)^\kappa$ for some $\kappa>0$ and that $\widetilde{\mathcal{C}}_\rest(100)<\infty$.

We remark that in order to obtain Theorem \ref{bcttheorem} it is enough to use the $d$-linear Kakeya theorem (Theorem \ref{guthkak}) away from the endpoint; see \cite{BCT}.

\subsubsection*{The proof of Proposition \ref{KRimpliesR}}
Our approach is a minor reworking of an argument in \cite{BCT}, which in turn is multilinearisation of the key aspects of the linear analysis of Bourgain \cite{Bo}; see \cite{TVV} for a similar argument in a bilinear setting. It is helpful to draw parallels with the more elementary proof of Proposition \ref{kakkakkak}. We begin by localising the $L^{2/(d-1)}(B(0,R))$ norm to balls of radius $R^{1/2}$.

For notational convenience let $\mathcal{C}(R)=\widetilde{\mathcal{C}}_\rest(R)$.
Let $\phi$ be a nonnegative real-valued compactly-supported bump function on $\mathbb{R}^d$ with the property that $\widehat{\phi}$ is nonnegative and bounded below on the unit ball. For each $R\geq 1$ and $x\in\mathbb{R}^d$ let $\phi_{R^{1/2}}^x(\xi)=e^{-ix\cdot\xi}R^{d/2}\phi(R^{1/2}\xi)$, so that $\widehat{\phi_{R^{1/2}}^x}=\widehat{\phi}(R^{-1/2}(x-\cdot))$ is nonnegative and bounded below on $B(x,R^{1/2})$ uniformly in $x$. From the modulation-invariance of the inequality \eqref{enoughconst} and the definition of $\mathcal{C}$ above, we have
\begin{equation}
\Bigl\|\prod_{j=1}^d\widehat{\phi_{R^{1/2}}^x}\widehat{f}_j\Bigl\|_{L^{\frac{2}{d-1}}(B(x,R^{1/2}))}\leq \mathcal{C}(R^{1/2}) R^{-\frac{d}{2}}\prod_{j=1}^d\|\phi_{R^{1/2}}^x*f_j\|_2
\end{equation}
uniformly in $x$.\footnote{On a minor technical note, this is where the choice of constant $c$ in the definition of $A_j(R)$ is relevant. More specifically, $c$ should be chosen such that the bump function $\phi$ satisfies $\supp (\phi_{R^{1/2}}^x*f_j)\subseteq A_j(R^{1/2})$.} Integrating this over $x\in B(0,R)$ we obtain
\begin{equation}\label{49}
\Bigl\|\prod_{j=1}^{d}\widehat{f}_{j}\Bigl\|_{L^{\frac{2}{d-1}}(B(0,R))}
\lesssim \mathcal{C}(R^{1/2})\Bigl(
R^{-d/2}\int_{B(0,R)}\Bigl(\prod_{j=1}^{d}
\|\phi_{R^{1/2}}^x*f_{j}\|_{L^2({\bf R}^d)}^{2}\Bigr)^{\frac{1}{d-1}}dx\Bigr)^{\frac{d-1}{2}}.
\end{equation}
Now for each $1\leq j\leq d$ we cover $A_{j}(R)$ by a boundedly overlapping
collection of discs $\{\rho_{j}\}$ of diameter $R^{-1/2}$, and
set $f_{j,\rho_j}:=\chi_{\rho_{j}}f_{j}$.
Since (for each $j$) the supports
of the functions $\phi_{R^{1/2}}^x*f_{j,\rho_j}$ have bounded overlap, it follows using Plancherel's theorem that
\begin{equation}\label{50}
\Bigl\|\prod_{j=1}^{d}\widehat{f}_{j}\Bigl\|_{L^{\frac{2}{d-1}}(B(0,R))}
\lesssim \mathcal{C}(R^{1/2})\Bigl(
R^{-d/2}\int_{B(0,R)}\Bigl(\prod_{j=1}^{d}
\sum_{\rho_{j}}\|\widehat{\phi_{R^{1/2}}^{x}}\widehat{f}_{j,\rho_j}\|_{L^2({\bf R}^d)}^{2}
\Bigr)^{\frac{1}{d-1}}
dx\Bigr)^{\frac{d-1}{2}}.
\end{equation}
Since the function $\widehat{\phi_{R^{1/2}}^{x}}$ is rapidly decreasing away from
$B(x,R^{1/2})$, in order to prove the proposition it is enough to show that
\begin{equation}\label{51}
\Bigl(
R^{-d/2}\int_{B(0,R)}\Bigl(\prod_{j=1}^{d}
\sum_{\rho_{j}}\|\widehat{f_{j,\rho_j}}\|_{L^{2}(B(x,R^{1/2}))}^{2}
\Bigr)^{\frac{1}{d-1}}dx\Bigr)^{\frac{d-1}{2}} \lesssim \mathcal{C}_{\kak}(R^{-1/2})R^{-\frac{d}{2}}\prod_{j=1}^d\|f_j\|_2.
\end{equation}
The point here is that the portions of $\widehat{f}_{j,\rho_j}$ on translates of $B(x,R^{1/2})$ in \eqref{50} can be handled by the modulation-invariance of the estimate \eqref{51}.

For each $\rho_j$ let $\psi_{\rho_j}$ be a Schwartz
function which is comparable
to $1$ on $\rho_j$ and whose compactly-supported Fourier transform satisfies
$$
|\widehat{\psi_{\rho_j}}(x+y)|\lesssim R^{-(d+1)/2}\chi_{\rho_{j}^*}(x)
$$
for all $x,y \in \R^d$ with $|y| \leq R^{1/2}$, where $\rho_{j}^*$ denotes an $O(R)\times
O(R^{1/2})\times \cdots\times O(R^{1/2})$-tube, centred at the
origin, and with long side pointing in the direction normal to the
disc $\rho_{j}$.
If we define
$\tilde{f}_{j,\rho_j}:=f_{j,\rho_j}/\psi_{\rho_j}$, then $f_{j,\rho_j}$
and $\tilde{f}_{j,\rho_j}$ are pointwise comparable, and furthermore
by Jensen's inequality,
$$|\widehat{f}_{j,\rho_j}(x+y)|^{2}=|\widehat{\tilde{f}}_{j,\rho_{j}}*
\widehat{\psi}_{\rho_j}(x+y)|^{2}\lesssim
R^{-(d+1)/2}|\widehat{\tilde{f}}_{j,\rho_{j}}|^{2}*\chi_{\rho_{j}^*}(x)
$$
whenever $x \in \R^d$ and $|y| \leq R^{1/2}$.  Integrating this in $y$ we conclude
$$
\|\widehat{f}_{j,\rho_j}\|_{L^{2}(B(x,R^{1/2}))}^{2}\lesssim
R^{-1/2}|\widehat{\tilde{f}}_{j,\rho_{j}}|^{2}*\chi_{\rho_{j}^*}(x).
$$
Combining this with \eqref{guthest3} gives
\begin{eqnarray*}
\begin{aligned}
\Bigl(
R^{-d/2}&\int_{B(0,R)}\Bigl(\prod_{j=1}^{d}
\sum_{\rho_{j}}\|\widehat{f_{j,\rho_j}}\|_{L^{2}(B(x,R^{1/2}))}^{2}
\Bigr)^{\frac{1}{d-1}}dx\Bigr)^{\frac{d-1}{2}}\\
&\lesssim
\Bigl(
R^{-d/2}\int_{B(0,R)}\Bigl(\prod_{j=1}^{d}
\sum_{\rho_{j}}R^{-1/2}|\widehat{\tilde{f}}_{j,\rho_{j}}|^{2}*\chi_{\rho_{j}^*}(x)
\Bigr)^{\frac{1}{d-1}}dx\Bigr)^{\frac{d-1}{2}}\\
&\lesssim \mathcal{C}_\kak(R^{-1/2})R^{-d/2}\prod_{j=1}^{d}
\Bigl(\sum_{\rho_{j}}\|\tilde{f}_{j,\rho_{j}}\|_{L^2(A^R_j)}^{2}\Bigr)^{1/2}\\
&\lesssim \mathcal{C}_\kak(R^{-1/2})R^{-d/2}\prod_{j=1}^{d}\|f_{j}\|_{L^2(A^R_j)}.
\end{aligned}
\end{eqnarray*}
In the last two lines we have used Plancherel's theorem,
disjointness, and the pointwise comparability of
$\tilde{f}_{j,\rho_j}$ and $f_{j,\rho_j}$. This completes the
proof of \eqref{51} and thus Proposition \ref{KRimpliesR}.

\subsection{From multilinear to linear: pointers to applications}\label{fourfour}
Recently Bourgain and Guth \cite{BG} developed a method by which $d$-linear restriction inequalities (in particular, Theorem \ref{bcttheorem}) may be used to make new progress on the linear restriction conjecture. We have already discussed a simple version of this method in Section \ref{twothree}. The idea is to find a suitable ``multilinear" analogue of Proposition \ref{bilBG}, taking the form
$$
|\widehat{gd\sigma}(\xi)|^q\lesssim K^{\gamma}\sum_{S_{\alpha_1},\hdots, S_{\alpha_d}\;\trans}|\widehat{g_{\alpha_1}d\sigma}(\xi)\cdots\widehat{g_{\alpha_d}d\sigma}(\xi)|^{\frac{q}{d}}+\cdot\cdot\cdot$$
for some $\gamma>0$.
This inequality, and its subsequent analysis, require several additional ingredients in order to bootstrap away all but the first term. Unfortunately there is some loss in doing this which places some limit on the resulting progress; see \cite{BG}.

For further applications of Theorem \ref{bcttheorem} and the Bourgain--Guth method, see \cite{BourgainMoment}, \cite{BourgainSch}, \cite{BourgainLap}, \cite{Temur} and \cite{LV}.
\section{Transversal multilinear harmonic analysis: a bigger picture}\label{five}

The aim of this section is to begin to investigate transversal multilinear analogues of problems in harmonic analysis related to curvature in a somewhat broader setting. As Stein pointed out in his 1986 ICM address \cite{SteinICM}, the analytical exploitation of underlying geometric properties such as nondegenerate curvature is intimately connected with the estimation of \emph{oscillatory integrals}. Looking at oscillatory integrals in some generality would thus seem to be a sensible place to begin.
\subsection{Oscillatory integrals}
To a smooth phase function $\Phi:\mathbb{R}^{d'}\times\mathbb{R}^d\rightarrow\mathbb{R}$ we may associate an operator
\begin{equation}\label{oscintdef}
T_\lambda f(\xi)=\int_{\mathbb{R}^{d'}} e^{i\lambda\Phi(x,\xi)}\psi(x,\xi)f(x)dx.
\end{equation}
Here $d'\leq d$ and $\psi$ is a smooth cutoff function on $\mathbb{R}^{d'}\times\mathbb{R}^d$.
Such operators are referred to as \emph{oscillatory integrals of H\"ormander type} or \emph{oscillatory integrals of the second kind}.

It is natural to look for $L^p-L^q$ control of $T_\lambda$ in terms of the large parameter $\lambda$ under nondegeneracy conditions on the phase $\Phi$.
The starting point in this well-studied problem is the classical H\"ormander theorem.
\begin{theorem}[H\"ormander]\label{Hormander}
If $d'=d$ and
\begin{equation}\label{hess}
\det\Bigl(\frac{\partial^2\Phi(x,\xi)}{\partial x_i\xi_j}\Bigr)\not=0 \;\;\;\;\;\;\bigl(\mbox{i.e.}\;\;\det\hess(\Phi)\not=0\bigr)$$ on $\supp(\Phi)$ then $$\|T_\lambda f\|^2_{L^2(\mathbb{R}^d)}\lesssim\lambda^{-d}
\|f\|^2_{L^{2}(\mathbb{R}^d)}.
\end{equation}
\end{theorem}
H\"ormander's theorem is the point of departure for much of the general (linear) theory of oscillatory integrals. In the situation where $d'=d$ we refer the reader to \cite{Seeger} for a survey of a variety of results relating to $L^2$ bounds in the situation where the condition \eqref{hess} fails and higher order nondegeneracy conditions are placed on the phase. In such situations an $L^p$ theory also becomes relevant; see for example \cite{GS2} or \cite{BS}.
When $d'<d$, and in particular when $d'=d-1$, the class of oscillatory integral operators of the form \eqref{oscintdef} contains the Fourier extension operators discussed in Section 2. As was raised by H\"ormander in \cite{Hormanderconjecture}, one might expect the extension estimates described in Section 2 to be stable under smooth perturbations of the phase functions $(x,\xi)\mapsto\xi\cdot\Sigma(x)$ in \eqref{extn}. While this is true for $d=2$ (\cite{Hormanderconjecture}) and up to the Stein--Tomas exponent (\cite{beijing}), Bourgain showed in \cite{BoHo} that it is in general false.

Our goal here is to move from a linear setting to a multilinear setting, replacing curvature hypotheses with appropriate ``transversality" hypotheses. It is natural to begin by looking for a suitable general setting in which to place the $d$-linear restriction estimates from Section \ref{MultSect}.
By contrast with the linear situation, the $d$-linear restriction theorem (Theorem \eqref{bcttheorem}) does turn out to be quite stable under smooth perturbations of the phase functions $(x,\xi)\mapsto \xi\cdot\Sigma_j(x)$. The appropriate transversality condition on the general phase functions $\Phi_1,\hdots,\Phi_d:\mathbb{R}^d\times\mathbb{R}^{d-1}\rightarrow\mathbb{R}$ amounts to the assertion that the kernels of the mappings $d_\xi d_x\Phi_1,\hdots,d_\xi d_x\Phi_d$ span $\mathbb{R}^d$ at every point. In order to be more precise let
$$
X(\Phi_j):=\bigwedge_{\ell=1}^{d-1}\frac{\partial}{\partial
x_{\ell}}\nabla_{\xi}\Phi_j
$$
for each $1\leq j\leq d$; by (Hodge) duality we may interpret each $X(\Phi_j)$ as an $\mathbb{R}^d$-valued function on $\mathbb{R}^d\times\mathbb{R}^{d-1}$.
In the extension case where $\Phi_j(x,\xi)=\xi\cdot\Sigma_j(x)$, observe that $X(\Phi_j)(x,\xi)$ is simply a vector normal to the surface $S_j$ at the point $\Sigma_j(x)$. A natural transversality condition to impose on the general phases $\Phi_1,\hdots,\Phi_d$ is thus
\begin{equation}\label{difftrans}
\det\left(X(\Phi_{1})(x^{(1)},\xi),\hdots,X(\Phi_{d})(x^{(d)},\xi)\right)>\nu
\end{equation}
for all
$(x^{(1)},\xi)\in\supp(\psi_{1}),\hdots,(x^{(d)},\xi)\in\supp(\psi_{d})$. The oscillatory integral analogue of Theorem \ref{bcttheorem} is the following.
\begin{theorem}[B--Carbery--Tao \cite{BCT}]\label{BCToscthm}
If \eqref{difftrans} holds then for each
$\varepsilon>0$
there is a constant $C_\varepsilon>0$ for which
$$
\Bigl\|\prod_{j=1}^{d}T_{j,\lambda}f_{j}\Bigl\|_{L^{\frac{2}{d-1}}(\mathbb{R}^{d})} \leq
C_{\varepsilon}\lambda^{-d+\varepsilon}\prod_{j=1}^{d}\|f_{j}\|_{L^{2}(\mathbb{R}^{d-1})}
$$
for all $f_{1},\hdots,f_{d}\in L^{p}(\mathbb{R}^{d-1})$ and $\lambda>0$.
\end{theorem}
Just as in the extension case, Theorem \ref{BCToscthm} may be seen as a consequence of a certain recursive inequality involving a corresponding Kakeya-type bound, analogous to \eqref{restrec}. This Kakeya-type bound amounts to an analogue of Theorem \ref{guthkak} where the straight tubes $T_j$ are replaced by $\delta$-neighbourhoods of segments of smooth curves (see \cite{Wisewell} for a treatment of ``curvy Kakeya" problems in the classical linear setting). More specifically, for each $1\leq j\leq d$, let $\mathbb{T}_j$ denote a finite collection of subsets of $\mathbb{R}^d$ of the form
$$
\{\xi\in\mathbb{R}^d:|\nabla_x\Phi_j(a,\xi)-\omega|\leq \delta,\;(a,\xi)\in\supp(\psi_j)\},$$
where $a, \omega\in\mathbb{R}^{d-1}$. Let us suppose that \eqref{difftrans} holds and denote by $\mathcal{C}_\ckak(\delta)$ the smallest constant $C$ in the inequality
$$
\Bigl\|\prod_{j=1}^d\sum_{T_j\in\mathbb{T}_j}\chi_{T_j}\Bigr\|_{L^{\frac{1}{d-1}}(\mathbb{R}^d)}\leq C\delta^d\prod_{j=1}^d\#\mathbb{T}_j
$$
over all such families of tubes $\mathbb{T}_1,\hdots, \mathbb{T}_d$. Similarly we define $\mathcal{C}_{\osc}(\lambda)$ to be the smallest constant $C$ in the inequality
$$
\Bigl\|\prod_{j=1}^{d}T_{j,\lambda}f_{j}\Bigl\|_{L^{\frac{2}{d-1}}(\mathbb{R}^{d})} \leq
C\lambda^{-d}\prod_{j=1}^{d}\|f_{j}\|_{L^{2}(\mathbb{R}^{d-1})}
$$
over all smooth phase functions $\Phi_1,\hdots,\Phi_d$ satisfying \eqref{difftrans}.\footnote{Naturally there should also be some quantitative control on the smoothness of the phases $\Phi_j$ in the definitions of $\mathcal{C}_\ckak$ and $\mathcal{C}_\osc$. Rather that discuss this technical aspect here, we refer the reader to \cite{BCT} for the details and the subsequent uniform version of Theorem \ref{BCToscthm}.}
Of course Theorem \ref{BCToscthm} may be restated as $\mathcal{C}_\osc(\lambda)=O(\lambda^\varepsilon)$ for all $\varepsilon>0$. We establish this via two consecutive recursive inequalities. The first of these concerns bounds on $\mathcal{C}_\ckak(\delta)$.
\begin{prop}[\cite{BCT} revisited]\label{ckakrec} There exists a constant $c$, independent of $\delta$ such that
\begin{equation}\label{ckakrecursive}
\mathcal{C}_\ckak(\delta)\leq c\mathcal{C}_\kak(\delta^{1/2})\mathcal{C}_{\ckak}(\delta^{1/2})
\end{equation}
for all $0<\delta\ll 1$.
\end{prop}
The proof of Proposition \ref{ckakrec} is simply a reprise of the proof of Proposition \ref{kakkakkak}. The key point is that if $Q$ is a cube of side $\delta^{1/2}$, then $T_j\cap Q$ is contained in a straight tube of dimensions $O(\delta)\times\cdots\times O(\delta)\times O(\delta^{1/2})$, and so \eqref{kakdec1} continues to hold in this curvy setting provided $\delta'\sim \delta^{1/2}$.

Now, since Guth's endpoint $d$-linear Kakeya inequality amounts to $\mathcal{C}_\kak(\delta)\leq C$ for some constant $C$ independent of $\delta$, we have $\mathcal{C}_\ckak(\delta)\leq cC\mathcal{C}_{\ckak}(\delta^{1/2})$, which on iterating $O(\log\log(1/\delta))$ times (as in \eqref{iterrest}) gives $$\mathcal{C}_\ckak(\delta)=O(\log(1/\delta))^\kappa$$ for some $\kappa$.

The second recursive inequality is analogous to Proposition \ref{KRimpliesR} and allows us to transfer the above bound on $\mathcal{C}_\ckak(\delta)$ to a bound on $\mathcal{C}_\osc(\lambda)$. This essentially tells us that $\mathcal{C}_\osc(\lambda)\leq c\mathcal{C}_\osc(\lambda^{1/2})\mathcal{C}_\ckak(\lambda^{-1/2}).$ The proof of this inequality, which is implicit in \cite{BCT} proceeds via a wavepacket decomposition in the spirit of \cite{Bo}. For technical reasons, relating to the fact that wavepackets cannot be perfectly localised to tubes, the actual inequality that follows is slightly more complicated.
\begin{prop}[\cite{BCT} revisited]\label{oscosckak}
For every $\varepsilon,M>0$ there exists constants $C_{\varepsilon,M}, c>0$ independent of $\lambda$ such that
$$\mathcal{C}_\osc(\lambda)\leq c\lambda^\varepsilon\mathcal{C}_\osc(\lambda^{1/2})\mathcal{C}_\ckak(\lambda^{-1/2+\frac{2\varepsilon}{d(d+1)}})+C_{\varepsilon,M}\lambda^{-M}.$$
\end{prop}
We leave the details of how this leads to Theorem \ref{BCToscthm} to the interested reader.



We conclude this subsection on oscillatory integrals by setting the above multilinear oscillatory integral inequalities in a much broader framework.
Since we do not intend curvature to feature in our analysis, it would seem natural to stay in the context of $L^2$ estimates. A rather expansive setting would be to consider inequalities of the form
\begin{equation}\label{MOI}
\int_{\mathbb{R}^d}\prod_{j=1}^k|T_{j,\lambda}f_j|^{2p_j}\lesssim\lambda^{-\alpha}\prod_{j=1}^k\|f_j\|_{L^{2}(\mathbb{R}^{d_j})}^{2p_j},
\end{equation}
where the $T_{j,\lambda}$ are H\"ormander-type oscillatory integral operators associated to phase functions $\Phi_j:\mathbb{R}^{d_j}\times\mathbb{R}^{d}\rightarrow\mathbb{R}$, and the $p_j$ and $\alpha$ are
real exponents.



In the very broad setting of \eqref{MOI} the appropriate generalisation of the notion of \emph{transversality} is perhaps unclear. However, as we are considering $L^2$ estimates, there are certain ``ready-made" notions of transversality provided by the classical theory of the \emph{Brascamp--Lieb inequalities}. Indeed, in the special case where the phases $\Phi_j$ are nondegenerate bilinear forms\footnote{This argument in the setting of H\"ormander's theorem ($k=1$, $d_1=d$, $p_1=1$ and $\alpha=d$) simply reduces the inequality \eqref{MOI} to the $L^2$-boundedness of the Fourier transform; see \cite{Hormanderconjecture} or \cite{Stein}.} $\Phi_j(x,\xi)=\langle x,L_j\xi\rangle$, and $\alpha=\frac{1}{2}(\sum p_jd_j+d)$, a scaling argument reduces inequality \eqref{MOI} to
$$
\int_{\mathbb{R}^d}\prod_{j=1}^k\Bigl|\int_{\mathbb{R}^{d_j}}e^{ix\cdot L_j\xi}f_j(x)\psi_j(x/\lambda^{1/2},\xi/\lambda^{1/2})dx\Bigr|^{2p_j}d\xi
\leq C\prod_{j=1}^k\|f_j\|_{L^2(\mathbb{R}^{d_j})}^{2p_j}
$$
uniformly in $\lambda\gg 1$. Choosing $\psi_j$ such that $\psi_j(0,0)\not=0$ for each $1\leq j\leq k$, a limiting argument reveals that
$$
\int_{\mathbb{R}^d}\prod_{j=1}^k|\widehat{f}_j(L_j\xi)|^{2p_j}d\xi\leq C\prod_{j=1}^k\|f_j\|_{L^2(\mathbb{R}^{d_j})}^{2p_j},
$$
which by Plancherel's theorem in equivalent to
\begin{equation}\label{bl}
\int_{\mathbb{R}^d}\prod_{j=1}^k (f_j\circ L_j)^{p_j}\leq C\prod_{j=1}^k\left(\int_{\mathbb{R}^{d_j}}f_j\right)^{p_j}
\end{equation}
for nonnegative integrable functions $f_j$, $1\leq j\leq k$.
Inequality \eqref{bl} is the classical Brascamp--Lieb inequality with datum $(\mathbf{L},\mathbf{p})=((L_j),(p_j))$, and following \cite{BCCT}, we denote by $\mathbf{BL}(\mathbf{L},\mathbf{p})$ the smallest value of $C$ for which \eqref{bl} holds over all nonnegative inputs $f_j\in\mathbb{R}^{d_j}$, $1\leq j\leq k$. The inequality \eqref{bl} generalises several important inequalities in analysis, including the multilinear H\"older, Young's convolution and Loomis--Whitney inequalities.
The Brascamp--Lieb inequalities have important applications in convex geometry and have been studied extensively -- see the 2006 ICM survey article of Barthe \cite{Barthe}. It is interesting to reflect that the ``self-similarity" properties of $\mathcal{C}_\kak$ and $\mathcal{C}_\ckak$ captured by the recursive inequalities \eqref{kakrecursive} and \eqref{ckakrecursive} are really manifestations of a more primordial property enjoyed by the general Brascamp--Lieb functional
$$\mathbf{BL}(\mathbf{L},\mathbf{p}; \mathbf{f}):=\frac{\int_{\mathbb{R}^d}\prod_{j=1}^m
(f_j\circ
L_j)^{p_j}}{\prod_{j=1}^m\left(\int_{\mathbb{R}^{d_j}}f_j\right)^{p_j}},$$
where $\mathbf{f}:=(f_j)$. A corresponding recursive inequality in this setting which has proved very useful, due to K. Ball \cite{Ball}, states that given two inputs $\mathbf{f}$ and $\mathbf{f}'$,
\begin{equation} \label{e:Ballinequality}
\mbox{BL}(\textbf{L},\textbf{p};\textbf{f})\mbox{BL}(\textbf{L},\textbf{p};\textbf{f}')
\leq \sup_{x\in\mathbb{R}^d}\mbox{BL}(\textbf{L},\textbf{p};(g_j^x))\;\mbox{BL}(\textbf{L},\textbf{p};\textbf{f}*\textbf{f}'),
\end{equation}
where $g_j^x(y) := f_j(L_jx-y)f'_j(y)$ and $\mathbf{f}*\mathbf{f}':=(f_j*f_j')$; see \cite{Ball}, \cite{BCCT} and \cite{BB}.


We may therefore interpret the oscillatory integral inequalities \eqref{MOI} with $\alpha=\frac{1}{2}(\sum p_jd_j+d)$ as certain ``oscillatory Brascamp--Lieb inequalities"\footnote{For closely related multilinear oscillatory integral inequalities, that may also be viewed as oscillatory Brascamp--Lieb inequalities, see \cite{ChLiThTa}, \cite{ChBL} and \cite{BCCTshort}.}. Indeed, one might tentatively conjecture that if $(\mathbf{L},\mathbf{p})$ is a Brascamp--Lieb datum for which $\mbox{BL}(\mathbf{L},\mathbf{p})<\infty$, and $\Phi_j:\mathbb{R}^{d_j}\times\mathbb{R}^d\rightarrow\mathbb{R}$ is smooth in a neighbourhood of the origin in $\mathbb{R}^{d_j}\times\mathbb{R}^d$ and satisfies $d_\xi d_x\Phi_j(0)=L_j$ for each $1\leq j\leq k$, then there is a finite constant $C$ for which \eqref{MOI} holds with $\alpha=\frac{1}{2}(\sum p_jd_j+d)$.\footnote{As the hypothesis on the phase is local, the cutoff functions $\psi_j$ would of course need to have suitably small supports.} The issue of finiteness of the general Brascamp--Lieb constant $\mbox{BL}(\mathbf{L},\mathbf{p})$ is quite nontrivial. As is shown in \cite{BCCT}, $\mbox{BL}(\mathbf{L},\mathbf{p})<\infty$ if and only if the scaling condition $\sum p_jd_j=d$ holds and
\begin{eqnarray}\label{dimension}
\dim(V) \leq \sum_{j=1}^k p_j \dim(L_j V) \hbox{ for all subspaces } V \subseteq \mathbb{R}^d,
\end{eqnarray}
and so in particular, this general set-up stipulates that $\alpha=d$ in \eqref{MOI}. The finiteness condition \eqref{dimension} is somewhat difficult to interpret in general, although under certain additional geometric hypotheses on the data $(\mathbf{L},\mathbf{p})$ it becomes much more transparent. A particularly simple example is the affine-invariant ``basis" hypothesis
\begin{equation}\label{strongtrans}
\bigoplus_{j=1}^k \ker L_j=\mathbb{R}^d,
\end{equation}
under which $\mbox{BL}(\mathbf{L},\mathbf{p})<\infty$ if and only if $p_j=\frac{1}{k-1}$ for all $1\leq j\leq k$. The basic example of data satisfying these strong conditions are those of the (affine-invariant) Loomis--Whitney inequality discussed in Section \ref{twotwo}. Indeed the proof of the affine-invariant Loomis--Whitney inequality alluded to in Section \ref{twotwo} (see also \cite{Finner}) reveals that
\begin{equation}\label{affinefinner}
\mathbf{BL}(\textbf{L},\textbf{p})=
\left|\star\bigwedge_{j=1}^k\star
Z_{d_j}(L_j)\right|^{-\frac{1}{k-1}},
\end{equation}
where $Z_{d_j}(L_j)\in\Lambda^{d_j}(\mathbb{R}^d)$ denotes the wedge product of the rows of the $d_j\times d$ matrix $L_j$, and $\star$ the Hodge star.
For an in-depth treatment of the Brascamp--Lieb inequalities see \cite{BCCT} and the references there.

The oscillatory Brascamp--Lieb inequalities \eqref{MOI} with the optimal exponent $\alpha=d$ appear to be extremely difficult. We observe that the oscillatory version of the tautological Brascamp--Lieb inequality obtained by setting $k=1$ is precisely H\"ormander's theorem (Theorem \ref{Hormander}), and the oscillatory Loomis--Whitney inequality contains the unresolved endpoint $d$-linear restriction conjecture. However, if we impose a strong structural condition, such as \eqref{strongtrans}, we are able to obtain \eqref{MOI} with a near-optimal exponent $\alpha$. Let us see how this leads to a a simple generalisation of Theorem \ref{BCToscthm}. Let us suppose that we have the spanning condition
\begin{equation}\label{spanning}
\bigoplus_{j=1}^k \ker d_xd_\xi\Phi_j(x^{(j)},\xi)=\mathbb{R}^d
\end{equation}
for all points $(x^{(1)},\xi)\in\supp(\psi_{1}),\hdots,(x^{(d)},\xi)\in\supp(\psi_{d})$. In order to quantify \eqref{spanning} let
$$
X_{d_j}(\Phi_j):= Z_{d_j}(d_xd_\xi\Phi_j)=\bigwedge_{\ell=1}^{d_j}\frac{\partial}{\partial
x_{\ell}}\nabla_{\xi}\Phi_j
$$
for each $1\leq j\leq k$. Invoking duality again (explicitly this time) we see that $\star X_{d_j}(\Phi_j)$ is a $\Lambda_{d-d_j}(\mathbb{R}^d)$-valued function, and thus $\star\bigwedge_j \star X_{d_j}(\Phi_j)$ is a well-defined real-valued function quantifying the hypothesis \eqref{spanning}. This allows us to substitute \eqref{spanning} for
\begin{equation}\label{spanningq}
\Bigl|\star\bigwedge_{j=1}^k \star X_{d_j}(\Phi_j)(x^{(j)},\xi)\Bigr|>\nu
\end{equation}
for some fixed $\nu>0$.
\begin{theorem}\label{BCTOBL}
If
\eqref{spanningq} holds
then for each $\varepsilon>0$ there exists a constant $C_\varepsilon<\infty$ such that
\begin{equation}\label{BCTOBLineq}
\int_{\mathbb{R}^d}\prod_{j=1}^k|T_{j,\lambda}f_j|^{\frac{2}{k-1}}\leq C_\varepsilon\lambda^{-d+\varepsilon}\prod_{j=1}^k\|f_j\|_{L^{2}(\mathbb{R}^{d_j})}^{\frac{2}{k-1}}
\end{equation}
for all $f_1\in L^2(\mathbb{R}^{d_1}),\hdots, f_k\in L^2(\mathbb{R}^{d_k})$.
\end{theorem}
Theorem \ref{BCTOBL} is a straightforward generalisation of Theorem \ref{BCToscthm}. We refer the reader to \cite{BCT} for the proof of Theorem \ref{BCToscthm} which readily extends to this setting.

The oscillatory Brascamp--Lieb inequalities \eqref{MOI} have considerable scope. In addition to encompassing the classical H\"ormander theorem, the classical Brascamp--Lieb inequalities and a variety of multilinear restriction/Kakeya-type inequalities, they also contain a broad family of multilinear Radon-like transforms that arise naturally in multilinear harmonic analysis and dispersive PDE.

\subsection{Transversal multilinear Radon-like transforms}
Restricting the oscillatory Brascamp--Lieb inequalities (\eqref{MOI} with $\alpha=d$) to phase functions of the form $\Phi_j(x,\xi)=x\cdot B_j(\xi)$, where $B_j$ is a nonlinear mapping from $\mathbb{R}^d$ to $\mathbb{R}^{d_j}$, leads us to the so-called \emph{nonlinear Brascamp--Lieb inequalities}, whereby the linear maps $L_j$ in \eqref{bl} are replaced with the nonlinear maps $B_j$. Indeed our tentative conjecture concerning the oscillatory Brascamp--Lieb inequalities becomes the following.
\begin{thma}[Nonlinear Brascamp--Lieb]
Let $(\mathbf{L},\mathbf{p})$ be a Brascamp--Lieb datum for which $\mathbf{BL}(\mathbf{L},\mathbf{p})<\infty$ and, for each $1\leq j\leq k$, let  $B_j:\mathbb{R}^{d}\rightarrow\mathbb{R}^{d_j}$ be a smooth submersion in a neighbourhood of $0\in\mathbb{R}^d$ with $dB_j(0)=L_j$. Then there exists a neigbourhood $U$ of $0$ such that
\begin{equation}\label{nlbl}
\int_{U}\prod_{j=1}^k (f_j\circ B_j)^{p_j}\lesssim
\prod_{j=1}^k\left(\int_{\mathbb{R}^{d_j}}f_j\right)^{p_j}.
\end{equation}
\end{thma}
Applications of such nonlinear inequalities typically require more quantitative statements. We do not concern ourselves with such matters here.

While obtaining such nonlinear inequalities would also appear to be rather difficult in full generality, there has been some progress under additional structural hypotheses on the data $(\mathbf{L},\mathbf{p})$. In particular, under the ``basis" hypothesis \eqref{strongtrans} the following theorem is a nonlinear version of the affine-invariant inequality \eqref{affinefinner}.
\begin{theorem}[\cite{BCW}, \cite{BB}]\label{t:main} Suppose
that $B_j:\mathbb{R}^d\rightarrow\mathbb{R}^{d_j}$ is a
smooth submersion
in a neighbourhood of a point
$0\in\mathbb{R}^d$ for each $1\leq j\leq k$. Suppose further that
\begin{equation} \label{e:directsumnonlin}
\bigoplus_{j=1}^k \ker dB_j(0)=\mathbb{R}^d
\end{equation}
and
\begin{equation}\label{radnu}
\Bigl|\star\bigwedge_{j=1}^k\star Z_{d_j}(dB_j(0))\Bigr| \geq \nu.
\end{equation}
Then there exists a neighbourhood $U$ of $0$, and a constant $C$ independent of $\nu$ such that
\begin{equation}\label{linadmisnl}
\int_{U}\prod_{j=1}^k (f_j\circ B_j)^{\frac{1}{k-1}}\leq
C\nu^{-\frac{1}{k-1}}\prod_{j=1}^k\left(\int_{\mathbb{R}^{d_j}}f_j\right)^{\frac{1}{k-1}}
\end{equation}
for all nonnegative $f_j\in L^1(\mathbb{R}^{d_j})$, $1\leq j\leq k$.
\end{theorem}
When $k=d$, Theorem \ref{t:main} reduces to a nonlinear version of the affine-invariant Loomis--Whitney inequality \cite{BCW}; see also \cite{BHT}. A more quantitative version of the general Theorem \ref{t:main} can be found in \cite{BB}, and a global version under the natural homogeneity assumption can be found in \cite{BBGshort}. There are similar results under higher-order hypotheses on the mappings $B_j$, in the case where their fibres are one-dimensional (i.e. $d_j=d-1$ for all $j$); see in particular \cite{TW} and \cite{Stovall}.

As may be expected given the numerous recursive inequalities that have featured so far (and in particular, Ball's inequality \eqref{e:Ballinequality}), Theorem \ref{t:main} may be proved by induction-on-scales. However, unlike all of the other results that we have deduced via a recursive inequality, the inequality \eqref{e:directsumnonlin} is \emph{sharp} in the sense that it would only follow from \eqref{BCTOBLineq} if we were able to remove the $\epsilon$-loss there. In short, we require a recursive inequality which, upon iterating, does not lead to an unbounded factor, such as a logarithm.
In order to achieve this we let $\mathcal{C}(\delta)$ denote the smallest constant $C$ in the inequality
\begin{equation}\label{linadmisnlind}
\int_{B(0,\delta)}\prod_{j=1}^k (f_j\circ B_j)^{\frac{1}{k-1}}\leq
C\nu^{-\frac{1}{k-1}}\prod_{j=1}^k\left(\int_{\mathbb{R}^{d_j}}f_j\right)^{\frac{1}{k-1}}
\end{equation}
holds over all smooth\footnote{As we have seen before, it is necessary that the smoothness ingredient in the definition of $\mathcal{C}(\delta)$ is suitably uniform; see \cite{BB}.} submersions $B_j$ satisfying \eqref{radnu} and input functions $f_j$. As is shown in \cite{BB}, there is a constant $\alpha>0$, depending on the smoothness of the $B_j$, for which the recursive inequality
\begin{equation}\label{radrec}
\mathcal{C}(\delta)\leq (1+O(\delta^\alpha))\mathcal{C}(\delta/2)
\end{equation}
holds for all $\delta>0$ sufficiently small. Since the factor $1+O(\delta^\alpha)$ converges to $1$ sufficiently quickly as $\delta\rightarrow 0$, iterating \eqref{radrec} leads to the desired bound $\mathcal{C}(\delta)\lesssim\nu^{-\frac{1}{k-1}}$ via a convergent product.
This efficient inductive approach is based on the closely-related \cite{BCT}. We note that there is a further precedent for this in \cite{TaoSharpCone}.

The nonlinear Brascamp--Lieb inequalities \eqref{nlbl} may be interpreted as estimates on certain \emph{multilinear Radon-like transforms}, and it is this perspective which provides the link with applications in harmonic analysis and dispersive PDE. Let us make this connection clear at least on a somewhat informal level.
Let us call a multilinear Radon-like transform a mapping $R$ of the form
$$Rf(x)=\int_{\mathbb{R}^{d_1}\times\cdots\times\mathbb{R}^{d_{k-1}}}f_1(y_1)\cdots f_{k-1}(y_{k-1})\delta(F(y,x))\psi(y,x)dy,$$
where $f=(f_j)_{j=1}^{k-1}$, $f_j:\mathbb{R}^{d_j}\rightarrow\mathbb{C}$ is a suitable test function, $x\in\mathbb{R}^n$ and $F:(\mathbb{R}^{d_1}\times\cdots\times\mathbb{R}^{d_{k-1}})\times\mathbb{R}^n\rightarrow\mathbb{R}^m$ is a suitably smooth function.
\footnote{Provided $\nabla F$ does not vanish on the support of $\psi$ then $\delta\circ F$ is a well-defined distribution.}
Notice that $Rf(x)$ may be interpreted as a surface integral (or ``average") of the tensor product $f_1\otimes\cdots\otimes f_{k-1}$ over the submanifold
$$
M_x=\{y\in\mathbb{R}^{d_1}\times\cdots\times\mathbb{R}^{d_{k-1}}: F(y,x)=0, (y,x)\in\supp(\psi)\},$$
which generically has dimension $d_1+\cdots+d_{k-1}-m$.
It is pertinent to seek natural nondegeneracy conditions on the function $F$, and exponents $p_1,\hdots,p_k$, for which $R$ extends to a bounded mapping from $L^{p_1}(\mathbb{R}^{d_1})\times\cdots\times L^{p_{k-1}}(\mathbb{R}^{d_{k-1}})$ into $L^{p_k'}(\mathbb{R}^n)$. In the linear setting (corresponding to $k=2$ here) such problems have been the subject of extensive study in recent years. In order to obtain nontrivial estimates (so-called ``$L^p$-improving") it is necessary that the nondegeneracy conditions imposed on $F$ capture a certain underlying ``curvature"; see for example the paper of Tao and Wright \cite{TW}. Our aim here is to focus on the multilinear setting, placing appropriate transversality, rather than curvature, conditions on $F$.
By duality, an $L^{p_1}(\mathbb{R}^{d_1})\times\cdots\times L^{p_{k-1}}(\mathbb{R}^{d_{k-1}})$ into $L^{p_k'}(\mathbb{R}^n)$ bound on $R$ may be expressed as a bound on a multilinear form
\begin{equation}\label{MR}\int_{\mathbb{R}^{d_1}\times\cdots\times\mathbb{R}^{d_{k}}}
\prod_{j=1}^{k}f_j(y_j)\delta(F(y))\psi(y)dy\lesssim\prod_{j=1}^{k}\|f_j\|_{L^{p_j}(\mathbb{R}^{d_j})},
\end{equation}
where we have re-labelled $n=d_k$.
By parametrising the support of the distribution $\delta\circ F$ we may write the above inequalities \eqref{MR} in the form
\begin{equation}\label{BL}\int_{\mathbb{R}^{d}}
\prod_{j=1}^{k}f_j(B_j(x))\psi(x)dx\lesssim\prod_{j=1}^{k}\|f_j\|_{L^{p_j}(\mathbb{R}^{d_j})}.
\end{equation}
Inequality \eqref{BL} is quickly converted into \eqref{nlbl} by replacing $p_j\in [1,\infty]$ with $1/p_j\in [0,1]$ and then $f_j$ with $f_j^{p_j}$.


A particularly palatable application of Theorem \ref{t:main} in the setting of such multilinear Radon-like transforms is the following.
\begin{cor}[\cite{BCW},\cite{BB}]
If
$G:(\mathbb{R}^{d-1})^{d-1}\rightarrow\mathbb{R}$ is a smooth function such that
$$
|\det(\nabla_{y_1}G(0),\ldots,\nabla_{y_{d-1}}G(0))| \geq
\nu,$$ then there exists a neighbourhood $V$ of the origin
in $(\mathbb{R}^{d-1})^{d-1}$, and a constant $C$ independent of $\nu$ such that
\begin{equation}\label{transrad}
\int_{V}g_1(y_1)\cdots g_{d-1}(y_{d-1})g_d(y_1+\cdots+y_{d-1})\delta(G(y))\,\mathrm{d}y \leq C\nu^{-\frac{1}{d-1}}\prod_{j=1}^d\|g_j\|_{(d-1)'}
\end{equation}
for all nonnegative $g_j \in L^{(d-1)'}(\mathbb{R}^{d-1})$, $1 \leq
j \leq d$.
\end{cor}
Inequality \eqref{transrad} is a convolution-type multilinear Radon-like transform estimate of the form \eqref{MR}
with $F(y)=(y_d-y_{d-1}-\cdots-y_1,G(y_1,\hdots,y_{d-1}))$. For a generalisation with symmetric, non-convolution type hypotheses on $F$, see \cite{BBG}.

A further corollary to Theorem \ref{t:main} is the following result concerning multiple convolutions of $L^p$ densities supported on smooth submanifolds $S_1,\hdots,S_k$ of $\mathbb{R}^d$. We shall suppose that $S_1,\hdots,S_k$ are \emph{transversal at the origin} in the sense that their normal spaces at the origin form a basis for $\mathbb{R}^d$.
\begin{cor}[\cite{BCW},\cite{BB}, \cite{BBG}]
If $1 \leq q \leq
\infty$ and $p_j' \leq (k-1)q'$, then
\begin{equation} \label{e:multcon}
\| f_1 d\sigma_1 * \cdots * f_k d\sigma_k \|_{L^q(\mathbb{R}^d)} \lesssim \prod_{j=1}^k \| f_j \|_{L^{p_j}(d\sigma_j)}
\end{equation}
for all $f_j \in L^{p_j}(d\sigma_j)$ supported in a
sufficiently small neighbourhood of the origin.
\end{cor}
We remark that setting $q=2$, $k=d$ and applying Plancherel's theorem in \eqref{e:multcon} leads to the sharp global $d$-linear restriction estimate
$$\|\widehat{g_1d\sigma_1}\cdots\widehat{g_dd\sigma_d}\|_{L^2(\mathbb{R}^d)}\lesssim \|g_1\|_{(2d-2)'}\cdots\|g_d\|_{(2d-2)'}.$$ While rather modest, this inequality does not follow from Theorem \ref{bcttheoremrevisted} due to the logarithmic factor present there.

Finally we remark that Theorem \ref{t:main} has been successfully applied to the well-posedness of the Zakharov system; see \cite{BHT}, \cite{BHHT} and \cite{BH}.

\subsection*{Further reading} There are a number of important further results in the work of Bourgain and Guth \cite{BG} which, for reasons of space, we have not discussed in this article.
In particular, the Bourgain--Guth method discussed in Sections \ref{twothree} and \ref{fourfour} continues to apply in the setting of H\"ormander-type oscillatory integrals and curvy Kakeya inequalities, yielding new (and sometimes sharp) linear estimates. We refer to \cite{BG} for further reading.

\bibliographystyle{amsplain}

\end{document}